\documentclass[11pt]{amsart}
\usepackage{amsmath}
\usepackage{amssymb, verbatim}
\usepackage{pstricks,pst-node,pst-plot}
\usepackage{comment}
\usepackage{color}
\usepackage{extarrows}
\usepackage{graphicx}
\usepackage{tikz}
\usetikzlibrary{arrows,positioning,shapes}
\usepackage{enumitem}  
\usepackage{mathtools}

\title[ The Deformed Hermitian-Yang-Mills Equation and  Level Sets of Harmonic Polynomials]{The Deformed Hermitian-Yang-Mills Equation and Level Sets of Harmonic Polynomials}

\author[A. Jacob]{Adam Jacob*}
\thanks{$^{*}$Supported in part by a Simons Collaboration Grant.}
 \address{Department of Mathematics, University of California Davis, 1 Shields Ave., Davis, CA, 95616}
 \email{ajacob@math.ucdavis.edu}

\numberwithin{equation}{section}

\renewcommand{\thanks}[1]{\footnote{#1}} % Use this for footnotes

\newcommand{\be}{\begin{equation}}
\newcommand{\bea}{\begin{eqnarray}}
\newcommand{\eea}{\end{eqnarray}} \newcommand{\ee}{\end{equation}}
 
 \def\ba{\begin{eqnarray}}
\def\ea{\end{eqnarray}}

\def\ra{\rightarrow}

\def\ti{\tilde}

\def\ra{\rightarrow}

\def\[{{\bf [}}
\def\]{{\bf ]}}

\begin{document}
\newtheorem{theorem}{Theorem}
\newtheorem{proposition}{Proposition}
\newtheorem{lemma}{Lemma}
\newtheorem{corollary}{Corollary}
\newtheorem{conjecture}{Conjecture}
\newtheorem{example}{Example}
\newtheorem{claim}{Claim}
\newtheorem{assumption}{Assumption}
\newtheorem{question}{Question}
\theoremstyle{plain}

\newtheorem{definition}{Definition}
\maketitle

%\vspace{.1in}

\begin{abstract}   Suppose $v(x,y):\mathbb C\rightarrow \mathbb R$ is an entire  harmonic polynomial with no critical points in the right half plane. Let $z_1, z_2\in\mathbb C$ lie on a level set of $v$ , and assume ${\rm Re}(z_2)>{\rm Re}(z_1)\geq0$. We give a necessary and sufficient condition, depending only on algebraic properties of the polynomial $v$, for when there exists a smooth real  function $f$ whose graph $x+if(x)$ lies on a level curve of $v$ connecting $z_1$ to $z_2$. Inspired by GIT, we construct a Kempf-Ness functional on an appropriate function space, and prove the functional is bounded from below and proper if and only if a such a graph exists. As an application, we find a stability condition equivalent to the existence of a solution to the deformed Hermitian-Yang-Mills equation on the family of projective bundles $ X_{r,m}:=\mathbb P(\mathcal O_{\mathbb P^m}\oplus \mathcal O_{\mathbb P^m}(-1)^{\oplus (r+1)})$ with Calabi Symmetry.\end{abstract}

\begin{normalsize}

\section{Introduction}
In this paper we study the deformed Hermitian-Yang-Mills equation on a family of projective bundles with large symmetry.  Known as {\it Calabi symmetry}, this condition allows us to write the equation as an ODE with prescribed boundary values, solutions of which lie on the level set of a given harmonic polynomial. Thus in order to solve our equation, we are led to the following more general question: When does  a level set of a  harmonic polynomial in $\mathbb C$ stays graphical over an interval in the $x$-axis? Surprisingly, even in this general setting, many of the formal structures from the original setup carry over, providing insight into the relationship between the background geometry and notions of algebraic stability.

First we describe the general problem. Let $v(x,y):\mathbb C\rightarrow \mathbb R$ be an entire  harmonic polynomial. For any $m\in\mathbb R$, we consider the level curve
$$\mathcal C_m:= \{x+iy\in\mathbb C\,|\, v(x,y)=m\}.$$
Now, choose two points $z_1=a +ip$ and $z_2=b+iq$ on the same level curve $\mathcal C_m$, with $a<b$. We are primarily interested in the following question:

\begin{question}
\label{mainquestion}
Under what conditions on the points $z_1, z_2,$ can we find a smooth function $f:[a ,b]\rightarrow\mathbb R$, with $f(a )=p, f(b)=q$, so that the graph  $x+if(x)\subset\mathbb C$ lies on  $\mathcal C_m?$  \end{question}

In order for such a graph to exist, we need   $z_1$ and $z_2$ to lie on the same connected component of $\mathcal C_m$.  Furthermore, for $ f(x)$ to be smooth, this component of $\mathcal C_m$ can not achieve vertical slope between $z_1$ and $z_2$. Now, in practice the level sets $\mathcal C_m$ can be quite complicated, and so we are interested in whether the above question has an answer which does not appeal to specific analytic details on the shape of  $\mathcal C_m$. In particular, a desirable solution is one  that depends only on $a, b, p, q,$ and algebraic properties of the polynomial $v$. We find such a simple condition using the Cauchy index, and furthermore relate existence of graph to notions of stability arising from  geometry.

More broadly, in complex differential geometry, many important PDEs  can be expressed as zeros of a moment map  associated to a complexified group orbit in an infinite dimensional symplectic manifold. Prominent examples include the work of Calabi \cite{Ca} on extremal metrics, Yau \cite{Yau} on Ricci-flat metrics, Chen-Donaldson-Sun \cite{CDS1,CDS2,CDS3} on K\"ahler-Einstein metrics on Fano manifolds, Atiyah-Bott \cite{AB} on the Yang-Mills equation, as well as the work of Donaldson \cite{D1,D2} and Uhlenbeck-Yau \cite{UY} on the  Hermitian-Yang-Mills equation. Inspired by Geometric Invariant Theory (GIT), a solution of the PDE corresponds to an associated algebraic stability condition. To elaborate further, we briefly discuss the finite dimensional case here. Let  $K$ be a compact real Lie group with complexification $G$,  and suppose $G$ acts on a projective K\"ahler manifold $(X,\omega)$, with $K$ preserving $\omega$. A point $x\in X$ is GIT stable if and only if the orbit of $x$ is closed under all one parameter subgroups of $G$, by the Hilbert-Mumford criterion. Now, for each $x$ one can construct a $K$-invariant function $\Phi_x$ called the Kempf-Ness function, which is convex along one parameter subgroups of $G/K$, and has critical points at zeros of the moment map. The Kempf-Ness Theorem then characterizes stability in terms of the properties of the Kempf-Ness function: $x$ is semistable stable if and only if $\Phi_x$ is bounded below, and $x$ is stable if and only if $\Phi_x$ is proper on $G/K$. Thus returning to  the infinite dimensional case, finding and studying an associated Kempf-Ness function can be an important step towards solving the desired PDE.

Now, in the case of the special Lagrangian equation, an infinite dimensional GIT framework has been put in place by Thomas \cite{T} and Solomon \cite{S}. A  key   motivation for studying  the deformed Hermitian-Yang-Mills equation  (which we denote by dHYM)   is that under semi-flat SYZ mirror symmetry it can be derived as the mirror  to the special Lagrangian graph equation   \cite{LYZ}.  Recently, Collins-Yau \cite{CY2, CY}  have  developed the mirror  GIT picture for the dHYM equation. In their extraordinary paper, the authors construct a space  of potentials for the equation (that can be thought of as analogous to $G/K$), on which they define a Riemannian structure and compute the geodesic  equation.  Additionally they introduce a Kempf-Ness functional $\mathcal J$ which is convex along geodesics and has critical points solutions of the dHYM equation. Their main analytic result is to prove existence of weak $C^{1,\alpha}$ geodesics on this space. These geodesics are use to compute the limiting slope of $\mathcal J$ along model infinite rays, yielding a necessary algebraic stability condition for existence. 

Our contribution  is to demonstrate that much of the formal GIT picture considered by Collins-Yau carries over to the general setup  behind Question \ref{mainquestion}. This includes constructing an appropriate function space along with an associated Kempf-Ness functional. Now, the notion of stability we consider is not quite as complicated as the necessary condition found by Collins-Yau, perhaps due to the specific geometry underlining our setup. Instead our stability is more in line with the initial conjecture posited by the author, along with Collins and Yau, in \cite{CJY}.  Before we introduce this condition, we establish some necessary background.

First, by adding a constant to our harmonic polynomial $v:\mathbb C\rightarrow \mathbb R$, it suffices to consider Question \ref{mainquestion} on the zero level set $\mathcal C_0$. Choose a harmonic conjugate $u$ and corresponding holomorphic function $w(z)=u+i v$. Let ${\mathcal H}_x=\{z\,|\, {\rm Re}(z)>0\}$ denote the right half-plane. We restrict our study to the case that $w$ has no critical points in ${\mathcal H}_x$, and the points $z_1, z_2$ lie in $\overline{{\mathcal H}_x}$. Thus one of our points can be a critical point, but not both. For simplicity we denote the choice of boundary data as $\vartheta:=\{z_1,z_2\}$. By a slight abuse of terminology,  for our initial definition of stability, we use the analytic condition of existence of a graph connecting $z_1$ to $z_2$. 

\begin{definition}
\label{stabledef}
\normalfont
For any boundary data $\vartheta\subset\overline{{\mathcal H}_x}$ we say
 \begin{enumerate}[label=(\roman*)] 
		\item    $\vartheta$ is {\it stable} if  there exists a a smooth function $f:[a,b]\rightarrow\mathbb R$ lying on a level curve of $\mathcal C_0$ and satisfying   $f(a)=p$, $f(b)=q$.
		\item    $\vartheta$ is {\it strictly semistable}  if  there exists a continuous function $f:[a,b]\rightarrow\mathbb R$ lying on a level curve of $\mathcal C_0$ and satisfying $f(a)=p$, $f(b)=q$, with $f$  not  $C^1$  at   $a$. 
		\item   $\vartheta$ is {\it unstable} if   no such continuous function exists.  
	\end{enumerate}
	In both cases (i) and (ii) we say $\vartheta$ is semistable. 
	\end{definition}

From this definition we turn to the task of finding a purely  algebraic criteria for existence. As a first step, we consider the level curves $$\mathcal D_0:=\{z\,|\, {\rm Re}(w'(z))=0\},$$ and demonstrate that ${\mathcal H}_x\backslash \mathcal D_0$ can be   partitioned into $n$ regions $A_1,...,A_n$. Next we define a counting function $N:\overline{{\mathcal H}_x}\longrightarrow \frac12\mathbb Z$ that is able to determine whether a point in $\overline{{\mathcal H}_x}$ lies inside or on the boundary of a particular region. This counting function only depends on algebraic properties of $w'(z)$. Specifically, if we fix a point $z_0=x_0+iy_0\in \overline{{\mathcal H}_x},$   then $N(z_0)$ is computed using the Cauchy index over $(-\infty, y_0]$ of the real polynomial 
\be
R_{x_0}(t)=\frac{{\rm Im}(w'(x_0+it))}{{\rm Re}(w'(x_0+it))}.\nonumber
\ee
By Sturm's Theorem, the Cauchy index can be computed from $R_{x_0}(t)$ using a division algorithm \cite[Theorem 3.20]{E}. Thus,  $N(z_1)$ and $N(z_2)$ depend on algebraic properties of the one variable polynomials $w'(a+iy)$ and $w'(b+iy)$. In the case that $a=0$ we also need the location and order of the critical points of $w$ on the $y$-axis. 

Our next step is to determine which initial configurations lead to graphs on the level set $\mathcal C_0$. There are a few cases to consider, as extra care is required when $z_1$  is a critical point of $w(z)$. Furthermore, a critical point may either be {\it generic} or {\it non-generic} (see Section  \ref{prelimsection} for a definition), and these cases need to be treated separately. Our first main result characterizes stability purely  in terms of the counting function $N.$

\begin{theorem}
\label{Cauchystable}
Fix boundary data $\vartheta=\{z_1,z_2\}$, and consider the counting function $N:\overline{{\mathcal H}_x}\longrightarrow \frac12\mathbb Z$ defined in \eqref{countingfunction1} and \eqref{countingfunction2}.
\begin{enumerate}[label=(\roman*)] 
	\item  If $z_1$ is not a critical point of $w(z)$, then the configuration $\vartheta$ is stable if  $|N(z_1)-N(z_2)|=0, $ strictly semistable if  $|N(z_1)-N(z_2)|=\frac12$, and unstable otherwise.

 \item If $z_1$ is a generic critical point of order $k$, then  $\vartheta$ is stable if  $$N(z_2)\in\{ N(z_1)-k, N(z_1)-k+1,..., N(z_1)\},$$ and  unstable otherwise.

 \item If  $z_1$ is a non-generic critical point of order $k$, then $\vartheta$ is stable if 
$$N(z_2)\in\{ N(z_1)-k+\frac12, N(z_1)-k+\frac32,..., N(z_1)-\frac12\},$$
strictly semistable  if  
$$N(z_2)\in\{ N(z_1)-k-\frac12, N(z_1)+\frac12\},$$
and unstable otherwise.
  	\end{enumerate}
	\end{theorem} Combined with Definition \ref{stabledef}, we view the above result  as a satisfactory answer to Question \ref{mainquestion}, as it provides conditions for existence that depend only algebraic data related to the endpoints.

Our next goal is to find an appropriate space of real functions over $[a,b]$, with specified boundary data, over which we can define our Kempf-Ness functional. In Section \ref{functionspace} we define $\mathcal M_\vartheta$, which is a natural analogue to the space of functions considered by Collins-Yau \cite{CY2, CY} for the dHYM equation.  On $\mathcal M_\vartheta$ we construct a metric and write down the corresponding geodesic equation. We then define our    Kempf-Ness functional  $\mathcal J:\mathcal M_\vartheta\rightarrow\mathbb R$  by its derivative along a path of functions $f_t$ via
\be
\frac{d}{dt}\mathcal J(\dot f_t)=\int_a^b\frac{ \dot f_t}{\sigma}  {\rm Im}(w(x+if_t))dx.\nonumber
\ee
Here $\sigma:[a,b]\rightarrow\mathbb R$ is a fixed  function that plays the role of a background metric. Right away we see that if $f$ lies on $\mathcal C_0$ then ${\rm Im}(w(x+if))=0$, and so we are at a critical point for $\mathcal J$. We demonstrate that $\mathcal J$ is convex along geodesics, and then show that our notions of stability correspond precisely to asymptotic behavior of paths in     $\mathcal J$. Specifically, we prove:

 \begin{theorem}
 \label{functionalstability}
For all boundary data $\vartheta=\{z_1, z_2\}\subset \overline{{\mathcal H}_x}$:
\begin{enumerate}[label=(\roman*)] 
		\item  $\vartheta$ is stable if   $\mathcal M_\vartheta$ is non-empty and the  $\mathcal J$-functional is bounded from below and proper.

		 		\item   $\vartheta$ is strictly semistable if   $\mathcal M_\vartheta$ is non-empty and the  $\mathcal J $-functional is bounded from below   but not proper.
				
						\item  $\vartheta$ is unstable if   $\mathcal M_\vartheta$ is empty, or $\mathcal M_\vartheta$ is non-empty and the  $\mathcal J $-functional is not bounded from below. 
	\end{enumerate}
\end{theorem}

Here by proper we mean that   $\mathcal J(\cdot)$ achieves a unique global minimum on the set $\mathcal M_\vartheta$. Not only does this result provide another satisfactory answer to Question \ref{mainquestion}, but again, as our main application  is to the dHYM equation on manifolds with Calabi-Symmetry, the above Theorem gives further evidence that in the general case the analogous $\mathcal J$-functional   will have similar properties.

We now define the dHYM equation. Let $(X,\omega)$ be a compact K\"ahler manifold of complex dimension $n$, and $[\alpha]\in H^{1,1}(X,\mathbb R)$ a real cohomology class.  The   dHYM  equation seeks a representative $\alpha\in[\alpha]$ satisfying
\be
\label{dHYM1}
{\rm Im}(e^{- i\hat\theta}(\omega+i\alpha)^n)=0,
\ee 
where $e^{i\hat\theta}\in S^1$ is a fixed constant.  Initial attempts to solve equation \eqref{dHYM1} were undertaken in \cite{JY} and later \cite{CJY}, and relied on certain analytic assumptions, namely that the class $[\alpha]$ admitted a representative that satisfied a positivity condition.  Following the work of  Lejmi-Sz\'ekelyhidi and Collins-Sz\'ekelyhidi  on the $J$-equation \cite{CS, LS}, the  author, along with T.C. Collins and   S.-T. Yau, integrated the positivity condition along subvarieties to develop a necessary class condition for existence, and conjectured it was a sufficient condition as well  \cite{CJY}. The conjecture can be described as follows.

For any analytic subvariety $V\subset X$, define the following complex number, which we refer to as the {\it charge} of the subvariety:
\be
\label{centralcharge}
Z_V([\alpha])=-\int_Ve^{i\omega+\alpha}.
\ee
By convention we only integrate the term in the expansion of order ${\rm dim}(V)$. 
Both our notation and the inspiration for our stability come from   the work of  Douglas-Fiol-R\"omelsberger on $\Pi$-stability \cite{DFR}, as well as the work of Bridgeland on stability on triangulated categories \cite{Bridgeland}. 
Now, the main results of \cite{CJY} rely on an assumption referred to as {\it supercritical phase}, which assumes that the constant $\hat\theta$ can be lifted to $\mathbb R$ to lie within the interval $((n-2)\frac\pi 2,n\frac\pi2)$. Therefore we state the conjecture with this assumption:
\begin{conjecture}[Collins-J-Yau \cite{CJY}] \label{theconjecture}
The  cohomology class $[\alpha]\in H^{1,1}(X,\mathbb R)$ on a compact K\"ahler manifold $(X,\omega)$ admits a solution to the deformed Hermitian-Yang-Mills equation \eqref{dHYM1} (with supercritical phase) if and only if $Z(X)\neq 0$, and for all analytic subvarieties $V\subset X$,
\be
\label{stabconj}
{\rm Im}\left(\frac{Z_V([\alpha])}{Z_X([\alpha])}\right)>0.
\ee
\end{conjecture}
Slightly weaker versions of the above conjecture have recently been solved by Chen \cite{C1}, and Chu-Lee-Takahashi \cite{CLT}. In particular Chen solved the conjecture using a uniform positivity assumption for the associated integrals, while Chu-Lee-Takahashi assume the slightly weaker assumption of stability along a test family. Notably this proves the full conjecture in the projective case. We remark that both of their results rest on the supercritical phase assumption, and in the general setting there is not an agreed upon conjecture for when solutions to the dHYM equation exists, although there are some natural extensions.

We now restrict our attention to the given class of manifolds with large symmetry.  Let $E\rightarrow\mathbb P^m$ be the rank $r+1$ vector bundle  associated to  the locally free sheaf $\mathcal O_{\mathbb P^m}(-1)^{\oplus (r+1)}$, and consider the projective bundle $ X_{r,m}:=\mathbb P(\mathcal O_{\mathbb P^m}\oplus E).$
The vector space    $H^{1,1}(X_{r,m},\mathbb R)$ is spanned by two classes:  The pullback of the hyperplane divisor from $\mathbb P^m$, denoted $D_H$, and the the divisor at infinity $D_\infty\in|\mathcal O_{X_{r,m}}(1)|$.
Choose a K\"ahler form in the class $[\omega]=\xi_1 [D_H]+b [D_\infty]$ for $\xi_1, b>0$, and fix a class $[\alpha]=\xi_2 [D_H]+q [D_\infty]$ for any real $\xi_2, q$. Additionally, the projection $\pi: X_{r,m}\rightarrow\mathbb P^m$ admits a section (the zero section of the bundle), which we denote by $P$.

The manifolds $X_{r,m}$ admit enough symmetry to allow for a family of  K\"ahler metrics  defined by a function of a single real variable, which Calabi exploited to construct metrics with constant scalar curvature \cite{Ca}. As mentioned above, such metrics are said to satisfy {\it Calabi symmetry}, and have been studied  in various other settings, including Metric flips \cite{SY}, the long term behavior of the K\"ahler-Ricci flow \cite{Song, SW2, SW3, SW4}, and the $J$-equation \cite{FL}. Here we use Calabi symmetry to reduce the dHYM equation to an ODE with boundary values, and demonstrate that graphs of solutions lie on a level set of  a harmonic function function Im$(w(z)),$ where $w(z)$ is the antiderivative of $w'(z)=e^{-i\hat\theta} (\xi_1+i\xi_2+z )^m z^r$ satisfying $w(0)=0$. Thus Theorem \ref{Cauchystable} applies in this setting, and  our main job is to relate the Cauchy index and counting function $N$ to the charges associated to the subvarieties $D_\infty$ and $P$.

To accomplish this, we construct  paths,  $Z_{D_\infty}(t):[q,\infty)\rightarrow \mathbb C^*$ and $Z_P(t):[0,\infty)\rightarrow \mathbb C^*$ (defined by \eqref{Dinftylift} and \eqref{ZPlift}), that begin at the charges  $Z_{D_\infty}([\alpha])$ and $Z_P([\alpha])$,  winding counterclockwise about the origin as $t$ increases, approaching the negative real axis as $t\rightarrow\infty$. We also construct a lift of $Z_X([\alpha])$ using a path in $\mathbb C^*$. Thus we get well defined lifts of the arguments of $Z_X([\alpha])$, $Z_{D_\infty}([\alpha])$, and $Z_P([\alpha])$. Again, following the notation of \cite{DFR},   these lifts allow us to define the {\it grade} of a subvariety $V$ by
$$\phi_V([\alpha])=\frac1\pi{\rm arg}(Z_V([\alpha])).$$
We then  conclude the following:
\begin{theorem}
\label{mainCalabi}
On the projective bundle $X_{r,m}:=\mathbb P(\mathcal O_{\mathbb P^m}\oplus \mathcal O_{\mathbb P^m}(-1)^{\oplus (r+1)}),$ fix a  K\"ahler form $\omega\in \xi_1 [D_H]+b [D_\infty]$ with  Calabi symmetry, and a class $[\alpha]=\xi_2 [D_H]+q [D_\infty]$. Given charges $Z_X([\alpha])$, $Z_{D_\infty}([\alpha])$, $Z_P([\alpha])$ in $\mathbb C^*,$ there exists a way to lift the argument of each charge to $\mathbb R$ that depends only on $\xi_1,\xi_2, b, q, m,$ and $r$. By construction of the lifts it holds
$$\phi_X([\alpha])<\phi_{D_\infty}([\alpha])< \phi_X([\alpha])+1.$$
Furthermore, the class $[\alpha]$ admits  solution to the deformed Hermitian-Yang-Mills equation if and only if 
$$\phi_X([\alpha])<\phi_{P}([\alpha])< \phi_X([\alpha])+r+1$$
in the generic case, and 
$$\phi_X([\alpha])+1<\phi_{P}([\alpha])< \phi_X([\alpha])+r+1$$
in the non-generic case. 
\end{theorem}
Note  that when $r=0$, the non-generic case corresponds to a single curve in $\mathcal C_0$ that passes through the origin and achieves vertical slope there. Thus  no solution exists, as evidenced by the above inequality. Interestingly,  the behavior of the charges of  other subvarieties of $X_{r,m}$ are not needed for our proof of existence. We plan to investigate the reasons behind this  in future work. Also, we remark that when $r=0$, the space $X_{0,m}$ is the blowup of $\mathbb P^{m+1}$ at one point, a case the author, along with N. Sheu, considered separately in \cite{JS}. Additionally, N. Sheu considered the special case of $X_{1,1}$ in \cite{Sheu}.

The paper is organized as follows. In Section \ref{prelimsection} we prove some background results on the analytic structure of the level curves $\mathcal C_0$ and $\mathcal D_0$ which we will need  for later results. In Section \ref{stabsection} we introduce the Cauchy index, define our counting function $N$, and prove Theorem \ref{Cauchystable}. In Section \ref{functionspace} we define and analyze our function space $\mathcal M_\vartheta$, as well as write down the geodesic equation. In Section \ref{KNsection} we show our Kempf-Ness functional $\mathcal J$ is convex along geodesics, and prove Theorem \ref{functionalstability}. Finally, in Section \ref{DHYMsection} we turn to our projective bundles $X_{r,m}$ and prove Theorem \ref{mainCalabi}.

{\bf Acknowledgements.} The author would like to thank Tristan C. Collins and Eugene Gorsky for some helpful discussion, in addition to S.-T. Yau for his continued support and encouragement. This work was funded in part by a Simons collaboration grant.

\section{Preliminary results}
\label{prelimsection}

Fix $v:\mathbb C\rightarrow \mathbb R$, an entire  harmonic polynomial of degree $n$. By adding a constant to $v$ it suffices to consider Question \ref{mainquestion} on the zero level set $\mathcal C_0$. Choose a harmonic conjugate $u $ and corresponding holomorphic function $w(z) =u(x,y) +i v(x,y) $ on $\mathbb C$. Let ${\mathcal H}_x:=\{z\in\mathbb C\,|\, {\rm Re}(z)>0\}$ denote the right half plane, and let $Z$ be the set of critical points of $w$, that is $Z=\{z_0\in\mathbb C\,|\, w'(z_0)=0\}$. We assume that  $Z\cap{\mathcal H}_x=\varnothing$, so $w$ has no critical points on the right half plane. This assumption will be quite useful to us for understanding where the level sets of $v$ achieve vertical slope.

Now, any level set $\mathcal C_m$ has vertical slope at a point $z_0$  only if $v_y(z_0)=0$. By the Cauchy-Riemann equations $w'(z)=v_y+iv_x$, so we see $\mathcal C_m$ has vertical slope only where it intersects $$\mathcal D_0:=\{z\in\mathbb C\,|\, {\rm Re}(w'(z))=0\}.$$  As a result understanding the set $\mathcal D_0$ plays an important role in answering Question \ref{mainquestion}. The following lemma, applied to $w'(z)$, will help give us a much clearer picture of what $\mathcal D_0$ can look like.

\begin{lemma}
\label{vertical}
Let $P(z)=\beta\prod_k (z-\xi_k)^{d_k}$ be a holomorphic polynomial, where $\beta$ and $\xi_k$ are complex numbers. Assume that all the roots $\xi_k$ have non positive real part.  Then the level set $ \{z\,|\, {\rm Re}(P(z))=0\}$ does not attain vertical slope at any point in   ${\mathcal H}_x $.
\end{lemma}
\begin{proof}
 A level set of ${\rm Re}(P )$ through a point $z_0$ has vertical slope   only if  $0=\frac{\partial}{\partial y} {\rm Re}(P(z_0))={\rm Re}(P'(z_0)i)=-{\rm Im}(P'(z_0))$. Thus if in addition $z_0$ lies on $\{z\,|\, {\rm Re}(P(z))=0\}$ we have ${\rm Re}(P(z_0))=0$ and ${\rm Im}(P'(z_0))=0.$ As a result the quotient $P'(z_0)/P(z_0)$ is purely imaginary.  

Now, away from the zeros of $P$ the above quotient can be expressed as
\be
\frac{P'(z)}{P(z)}=\sum_k\frac{d_k}{z-\xi_k}=\sum_k\frac{d_k(\bar z-\bar \xi_k)}{|z-\xi_k|^2}.\nonumber
\ee
Writing $z_0=x_0+iy_0$, and taking the real part at this point we see
\be
0={\rm Re}\left(\frac{P'(z_0)}{P(z_0)}\right)= \sum_k\frac{d_k(x_0-{\rm Re}( \xi_k))}{|z_0-\xi_k|^2}.\nonumber
\ee
However, by assumption ${\rm Re}( \xi_k)\leq0$ for all $k$. Thus if $x_0>0$ it is impossible for the above terms to sum to zero.
\end{proof}

Thus we can conclude  $\mathcal D_0$ does not achieve vertical slope in ${\mathcal H}_x $. Now, outside of a large enough ball $B_R(0)$, the level set  $\mathcal D_0$ looks like $2(n-1)$ curves asymptotic to $2(n-1)$ rays with angle $\frac{\pi}{n-1}$ between each ray. As a result, generically  the intersection $\mathcal D_0\cap {\mathcal H}_x$   is asymptotic to $n-1$ rays with angle $\frac{\pi}{n-1}$, unless one of the asymptotic lines from $\mathcal D_0$ has vertical slope (in this case  $\mathcal D_0\cap {\mathcal H}_x$  will be asymptotic to $n-2$ rays). Since the curves in $\mathcal D_0$ do not intersect themselves nor have vertical slope in ${\mathcal H}_x$, it follows that $\mathcal D_0\cap {\mathcal H}_x$  consists of $n-1$ distinct curves, which we label as $\gamma_1,...,\gamma_{n-1}$, numbering counter-clockwise from the negative $y$-axis. Here we assume the generic case of $n-1$ curves, since in the case that there are only $n-2$ curves the only difference is there is one less curve to label.  Note that  if $w'(z)$ has a zero  along the $y-$axis, then the closures $\overline{\gamma}_1,...,\overline{\gamma}_{n-1}$ may intersect there, with the number of curves in the intersection depending on the order of the zero.

As a result   partition ${\mathcal H}_x\backslash \mathcal D_0$ into   $n$ distinct connected components, with the boundaries between each component given by the curves $\gamma_k$. Numbering counter-clockwise from the negative $y$-axis we call these components $A_1,...,A_n$. We now introduce the following important result:
\begin{proposition}
\label{onecurveregion}
For each $k$, the intersection $A_k\cap \mathcal C_0$ consists of exactly one distinct curve.
\end{proposition}

Before we prove this proposition, we first  elaborate on  the behavior of $\mathcal C_0$ for $|z|$ large. As implied by the discussion of  $\mathcal D_0$ above,  outside of a large ball, a holomorphic polynomial $w(z)$ is dominated by the leading order for $|z|$ large, and so the corresponding level set of the real or imaginary part will be asymptotic to $2n$ rays with angle $\frac{\pi}{n}$ between each ray. Here we explicitly find these asymptotic rays.

First, let $w(z)=\beta\prod_{k=1}^n (z-\xi_k)$, where $\beta$ is a fixed complex number and $\xi_k$ are the roots (here $\xi_k$ do not need to  be distinct). Multiplication by a nonzero constant will not affect $\mathcal C_0$, so it suffices to assume that $|\beta|=1$. The naive scaling $\lim_{r\rightarrow\infty}\frac1{r^n}w(rz)=\beta z^n$ gives a good   approximation for the behavior of $w(z)$ for $|z|$ large, however, because  all constant terms scale away, the level set $\{z\,|\, {\rm Im}(\beta z^n)=0\}$ will only be a translation of the asymptotic level set to $\mathcal C_0$. Specifically, for large $|z|$ the set $\mathcal C_0$ will be asymptotic to rays from the set $\{z\,|\, {\rm Im}(\beta (z-z_0)^n)=0\}$ for some fixed $z_0$, and we need to determine  this point $z_0=x_0+iy_0$.

Rotating the polynomial $\beta\prod_{k=1}^n (z-\xi_k)$ will not affect the point where the asymptotic rays emanate from, so in determining $z_0$ we instead look to the set  $\{z\,|\, {\rm Im}(\prod_{k=1}^n (z-\xi_k))=0\}$, which now contains a component asymptotic to a horizontal line. To determine this line, we look at the behavior of the  function $ \prod_{k=1}^n (z-\xi_k)$ for large $x$ along a horizontal line $y=y_0$. Specifically we have
\bea
 {\rm Im}\left(\prod_{k=1}^n (x+iy_0-\xi_k)\right)&=& x^{n-1} {\rm Im}\left(\sum_{k=1}^n (iy_0-\xi_k)\right)+O(x^{n-2})\nonumber\\
 &=&x^{n-1}\left(ny_0-\sum_{k=1}^n{\rm Im} (\xi_k)\right)+O(x^{n-2}).\nonumber
\eea
Here the highest order term has order ${n-1}$ since the $x^n$ term is purely real. For $x$ large this term dominates, and so the set where  ${\rm Im}(\prod_k (z-\xi_k))=0$ is asymptotic to the line $y=\frac1n\sum_{k=1}^n{\rm Im} (\xi_k)$, that is, $y_0$ is the average of the imaginary parts of all the roots of $w(z)$. Similarly it is not hard to see that $x_0$ needs to be average of the real parts of all the roots of $w(z)$. Thus $z_0$ needs to be the average of the roots of $w(z)$.

\begin{lemma}
Outside of a large enough ball, the level curves of $\mathcal C_0$ are asymptotic to the rays that make up the set $\{z\,|\, {\rm Im}(\beta (z-z_0)^n)=0\}$, where $z_0$ is the average of the roots of $w(z)$. Furthermore,  the level curves of $\mathcal D_0$ are asymptotic to the rays that make up the set $\{z\,|\, {\rm Re}(\beta (z-z_0)^{n-1})=0\}.$
\end{lemma}
\begin{proof}
We have already demonstrated that the level curves of $\mathcal C_0$ are asymptotic to the rays that make up the set $\{z\,|\, {\rm Im}(\beta (z-z_0)^n)=0\}$, thus we only need to show that the asymptotic rays corresponding to the derivate $w'(z)$ again emanate from the same point $z_0$.

As above, rotating $w(z)=\beta\prod_{k=1}^n (z-\xi_k)$ by $-{\rm arg}(\beta)$ will not affect $z_0$, so we instead consider the polynomial $P(z)=\prod_{k=1}^n (z-\xi_k)$. Taking the derivate gives $$P'(z)=\sum_{k=1}^n(z-\xi_1)\cdots(z-\xi_{k-1})(z-\xi_{k+1})\cdots (z-\xi_n).$$
Fixing $y$ at $y_0$, we again look at the behavior of the imaginary part of the above polynomial for large $x$. In particular 
\bea
 {\rm Im}\left(P'(x+iy_0)\right)&=& x^{n-2} {\rm Im}\left(\sum_{k=1}^n\sum_{j\neq k} (iy_0-\xi_j)\right)+O(x^{n-3})\nonumber\\
 &=&x^{n-2}\sum_{k=1}^n\left((n-1)y_0-\sum_{j\neq k}{\rm Im} (\xi_j)\right)+O(x^{n-3})\nonumber\\
 &=&x^{n-2} \left(n(n-1)y_0-(n-1)\sum_{k=1}^n{\rm Im} (\xi_k)\right)+O(x^{n-3}).\nonumber
\eea
Thus,  we see the level set is asymptotic to the line $y=\frac1n\sum_{k=1}^n{\rm Im} (\xi_k)$, and can again conclude $y_0$ is the average of the imaginary parts of the roots of $w(z)$. Similarly we conclude  $x_0=\frac1n\sum_{k=1}^n{\rm Re} (\xi_k)$. As a result the level set $\mathcal D_0$ is asymptotic to the  rays from the set $\{z\,|\, {\rm Re}(\beta (z-z_0)^{n-1})=0\},$ completing the proof of the lemma.

\end{proof}

With the above Lemma, we introduce the following notation and terminology. Given $w(z)=\beta\prod_{k=1}^n (z-\xi_k)$ we define 
$$T_{\mathcal C_0}^\infty:=\{z\,|\, {\rm Im}(\beta (z-z_0)^n)=0\}$$
to be the {\it tangent cone} to $\mathcal C_0$ at infinity. Again $z_0$ is the average of the roots of $w(z)$. Similarly let 
$$T_{\mathcal D_0}^\infty:=\{z\,|\, {\rm Re}(\beta (z-z_0)^{n-1})=0\}$$
 be the {\it tangent cone} to $\mathcal D_0$ at infinity. Additionally, we say a tangent cone is {\it non-generic} if it  contains a line of vertical slope. Otherwise we say it is {\it generic}.
 
 \begin{lemma}
 \label{genericlemma}
 $T_{\mathcal C_0}^\infty$ is generic if and only if $T_{\mathcal D_0}^\infty$ is.
 \end{lemma}
 \begin{proof} Without loss of generality set $z_0=0$. Suppose  $T_{\mathcal C_0}^\infty$ contains a line of vertical slope. Let $z=re^{i\theta}$ and $\beta=e^{i\phi_0}$. Then a point $z$ lies in $T_{\mathcal C_0}^\infty$  if and only if $n\theta+\phi_0=m\pi$ for some $m\in\mathbb Z$. Yet this implies
$$\frac1{r}{\rm Re}(\beta z^{n-1})={\rm cos}((n-1)\theta+\phi_0)={\rm cos}(m\pi-\theta)={\rm cos}(m\pi){\rm cos}(\theta).$$
Thus if $z$ lying on $T_{\mathcal C_0}^\infty$  also lies on the $y-$axis, ${\rm cos}(\theta)=0$ and thus $z$ lies on $T_{\mathcal D_0}^\infty$ as well. Similarly one can argue that a point $z$ lying on both   $T_{\mathcal D_0}^\infty$ and the $y$-axis must also lie on $T_{\mathcal C_0}^\infty$. 
\end{proof}

We now have the necessary background material to prove Proposition \ref{onecurveregion}. 
\begin{proof}[Proof of Propostion \ref{onecurveregion}]

Given the above explicit description of the tangent cones, we now know asymptotic rays for $\mathcal C_0$ alternate with those from $\mathcal D_0$, and that  the non-generic cases coincide. We first prove the generic case. The alternating condition tells us that each set from $A_1,...,A_n$ must contain one asymptotic ray from $\mathcal C_0$, and therefore each $A_k$ must contain at least a portion of a curve from $\mathcal C_0$. Denote this curve by $\Gamma_1$, which is asymptotic to a ray from $T_{\mathcal C_0}^\infty$.

Via contradiction, suppose a region $A_k$ contains another distinct curve that lies on  $\mathcal C_0$, denoted  $\Gamma_2$.  We  turn to Theorem 3.1-1 from \cite{B},  which states that a family $F$ of level curves of a harmonic function $v$ (in a simply connected  domain $D$) is a branched, regular curve
 family filling $D$,  where the critical points of $v $ are the branch points of $F$. Thus we know that $\Gamma_1\cap\Gamma_2=\varnothing$ in $A_k$, since otherwise there would be a branch point, and we know there are no critical points in ${\mathcal H}_x$. 
 
Note that neither $\Gamma_1$ nor $\Gamma_2$ can attain vertical slope at a point in $A_k$, since otherwise a curve from $\mathcal D_0$ would have to pass through this point, which is impossible because $A_k$ is bounded below by $\gamma_{k-1}$ and above by $\gamma_k$, with no curves from $\mathcal D_0$ in between.  In fact, this implies that the sign of $v_y$ does not change in $A_k$. Without loss of generality, assume that $v_y>0$ in $A_k$.

In the generic case there is only one asymptotic ray from $T_{\mathcal C_0}^\infty$ in $A_k$, and so $\Gamma_2$ must leave $A_k$ inside of some large ball $B_R(0)$. It can not leave $A_k$ on two points on the $y$-axis, since otherwise it would have a point of vertical slope in between. Thus    $\Gamma_2$  must leave  $A_k$ through $\gamma_{k-1}$ or $\gamma_{k}$. Denote this point of intersection by  $P=c+id$. We now consider several cases. First assume the vertical line $x=c$ intersects $\Gamma_1$ at a point $Q=c+id'\in \overline{A_k}$. Because both $P$ and $Q$ lie on $\mathcal C_0$, we know $v(P)=v(Q)=0$. Yet by Rolle's Theorem this implies there is a point $\zeta$ on the   interior of the segment $\overline{PQ}$ which satisfies $v_y(\zeta)=0$, contradicting the fact that $v_y>0$ in $A_k$.

Thus we can assume that the  vertical line $x=c$ does not intersect $\Gamma_1$, and so $\Gamma_1$ must leave $A_k$ at some point $Q=c'+id'$, with $c'>c$. Without loss of generality, assume that $\Gamma_2$ leaves $A_k$ at $P$ along the lower boundary $\gamma_{k-1}$. First we argue that $Q$ can not lie on $\gamma_{k-1}$ as well. To see this, suppose not, and use the fact that $v(P)=v(Q)=0$ implies the restriction of $v$ to $\gamma_{k-1}$ has a critical point, i.e. there exists a point $\zeta$ where $\gamma_{k-1}'\cdot \nabla v(\zeta)=0$. Because $\gamma_{k-1}$ does not achieve vertical slope, we can write $\gamma_{k-1}'=\kappa_1 \frac{\partial}{\partial x}+\kappa_2\frac{\partial}{\partial y}$ with $\kappa_1\neq0$. This implies
$$0=\gamma_{k-1}'\cdot \nabla v(\zeta)=\kappa_1 v_x(\zeta)+\kappa_2 v_y(\zeta)=\kappa_1 v_x(\zeta)$$
since $v_y=0$ along $\gamma_{k-1}$. Thus $v_x(\zeta)=0$, making $\zeta$ a critical point, a contradiction. We conclude $Q$ must lie on $\gamma_k$.

Now, let $\ti Q$ be the point where $\gamma_{k-1}$ intersects the line $x=c'$. Also, let $[t_1,t_2]$ be the time interval that parametrizes  $\gamma_{k-1}$ between $P$ and $\tilde Q$, i.e. $\gamma_{k-1}(t_1)=P$ and $\gamma_{k-1}(t_2)=\ti Q$. For each $t\in[t_1,t_2]$, denote by $\sigma_t$ the curve which lies on the segment  $\{x={\rm Re}(\gamma_{k-1}(t))\}\cap A_k.$ In particular, $\sigma_{t_2}$ represents the vertical segment connecting $\ti Q$ with $Q$. Because $v_y>0$ on $A_k$, and $v(P)=0$, it follows that $v>0$ on $\sigma_{t_1}$. Furthermore, because $v_y>0$ and $A_k$ and $v(Q)=0$, we must have $v<0$ on $\sigma_{t_2}$. Therefore, if we consider the function $L:[t_1, t_2]\rightarrow\mathbb R$ defined by
$$L(t)=\int_{\sigma_{t}}v(t,y)dy,$$
we have $L(t_1)>0$ and $L(t_2)<0$. Since $L(t)$ is continuous there exists a $t_0\in(t_1,t_2)$ with $L(t_0)=0$. Yet if this integral is zero, there must be a point $\zeta$ on the vertical segment $\sigma_{t_0}$ with $v(\zeta)=0$, so $\zeta\in\mathcal C_0$. Thus there exists some other curve $\Gamma_3$ through $\zeta\in A_k$ which lies on $\mathcal C_0$. However $\Gamma_3$ cannot pass over the lines $x=c$ and $x=c'$ in $A_k$, without creating a contradiction as before. Furthermore, it cannot pass over $\gamma_{k-1}$ or $\gamma_k$ without creating two points along either curve with $v=0$, again a contraction. Finally, we know $\Gamma_3$ can not be a closed curve by the maximum principle. Thus in all cases we get a contradiction, and so $A_k$ can not have two curves from $\mathcal C_0$. Certainly, the above arguments will also rule out the possibility that $A_k$ has more than two distinct curves from $\mathcal C_0$. We conclude that in the generic case there is exactly one curve from $\mathcal C_0$ in $A_k$.

We now consider the non-generic case. Here we have $n-1$ regions $A_1,...,A_{n-1}$. The main difference we encounter in this case is that while $A_1$ and $A_{n-1}$ still only have one asymptotic ray from $T_{\mathcal C_0}^\infty$ within the region, it is possible that outside of a large ball, $\mathcal C_0\cap B_R(0)^c\cap A_{\ell}$ (for $\ell=1,n-1$) could consist of two curves $\Gamma_1$ and $\Gamma_2$, with $\Gamma_1$ asymptotic to the $y-$axis and $\Gamma_2$ asymptotic to the ray with angle $\frac \pi n$ off the $y$-axis. For simplicity we work in $A_1$. Suppose that $\Gamma_1$  and $\Gamma_2$ are distinct curves. If $\Gamma_1$ is parameterized so $t\rightarrow-\infty$ corresponds to $\Gamma_1$ approaching the negative $y$-axis, it must then satisfy $\Gamma_1' \cdot\frac{\partial}{\partial x}>0$. That is, as $t$ increases $\Gamma_1$ heads in the positive $x$ direction. At no point can it turn away from this direction and head back to the $y-$axis without achieving positive slope. So if $\Gamma_1$ is distinct from $\Gamma_2$ it must leave $A_1$ through the boundary $\gamma_1$. However, we now arrive at a contradiction since $\Gamma_2$ can not cross $\Gamma_1$ without creating a critical point, nor can $\Gamma_2$ cross $\gamma_1$ without creating two points on $\gamma_1$ where $v=0$. Thus in this case $A_1$ can only admit one curve from $\mathcal C_0$, which starts off asymptotic to the $y$-axis, and then turns back down and heads out along the other asymptotic ray with angle $\frac \pi n$ off the $y$-axis. Note that it is possible for $A_1$ to have a curve which is  asymptotic ray with angle $\frac \pi n$ that then crosses $\gamma_1$, with no curves asymptotic to the $y$-axis, but in this case again $A_1$ only has one curve from $\mathcal C_0$. The rest of the regions also contain only one curve by arguing as before. This completes the proof of the proposition.

\end{proof}

In the arguments to follow we will need to understand the behavior of $\mathcal D_0$ and $\mathcal C_0$ near critical points on the $y-$axis. To this end we introduce the notion of a tangent cone to a critical point, which is a similar construction to the above tangent cone at infinity.

Let $iy_0$  be a critical point of order $k$. Write
 $$w'(z)=\beta(z-iy_0)^k\prod_\ell (z-\xi_\ell)^{d_\ell},$$ where $\beta$ is a fixed complex number and $\xi_\ell$ are the remaining critical points. Translating by $iy_0$,   scaling via $\frac 1 {r^k} w'(rz+iy_0)$, and taking the limit as $r\rightarrow 0$ allows us to write our rescaled level set as
$$T^{iy_0}_{\mathcal D_0}:=\{z\,|\, {\rm Re}(\beta z^k\prod_\ell (iy_0-\xi_\ell)^{d_\ell})=0\},$$
which we call the {\it tangent cone} to $\mathcal D_0$ at  $iy_0$. Notice $T^{iy_0}_{\mathcal D_0}$ consists of $k$ lines through the origin with angle $\frac{\pi}k$ between each line. For simplicity denote the constant $\beta  \prod_\ell (iy_0-\xi_\ell)^{d_\ell}$ by $\ti \beta$.  There are two types of critical points that can occur. Similar to the definition above,   we call a critical point $iy_0$ {\it non-generic} if $T^{iy_0}_{\mathcal D_0}$ contains a line of vertical slope. In this case $T^{iy_0}_{\mathcal D_0}$ can be expressed as $\{z\,|\, {\rm Re}( z^k )=0\}$ when $k$ is odd and $\{z\,|\, {\rm Re}(i z^k )=0\}$ when $k$ is even. When $T^{iy_0}_{\mathcal D_0}$ does not contain  a line of vertical slope we call $iy_0$ a {\it generic} critical point.  

Suppose    $iy_0$ is a critical point of order $k$ that also lies on   $\mathcal C_0$. Thus ${\rm Im}(w(iy_0))=0$. Even though ${\rm Re}(w(iy_0))$ may not be zero, shifting $w$ by $-{\rm Re}(w(iy_0))$ does not change the level set $\mathcal C_0$ (and also does not affect $\mathcal D_0$), so if we want to analyze the local  behavior of $\mathcal C_0$ near $iy_0$ we can assume without loss of generality that $iy_0$ is also a zero of $w$, and therefore a zero of order $k+1$. Furthermore, it is easy to check that if the tangent cone to $\mathcal D_0$ at $iy_0$ is given by $T^{iy_0}_{\mathcal D_0}= \{z\,|\, {\rm Re}(\ti\beta z^{k})=0\}$ as above, then then the tangent cone to $\mathcal C_0$ is given by
$$T^{iy_0}_{\mathcal C_0}= \{z\,|\, {\rm Re}(\frac1{k+1}\ti\beta z^{k+1})=0\}.$$
In particular, the tangent cone to $w'(z)$ can be understood by taking the derivative of the tangent cone to $w(z)$ in this case. Thus, the exact proof of Lemma  \ref{genericlemma} allows us to conclude: 
\begin{lemma} 
 Let $iy_0$ be a critical point of $w(z)$ that also lies on $\mathcal C_0$. Then $T_{\mathcal C_0}^{iy_0}$ is generic if and only if $T_{\mathcal D_0}^{iy_0}$ is.
 \end{lemma}
We also need the following result:
\begin{lemma}
\label{curveinlemma}
Suppose $iy_0$ is a non-generic critical point of order $k$. Then there exists a small $\delta>0$ such that  $ \mathcal D_0\cap B_{\delta}(iy_0)\cap {\mathcal H}_x$ consists of $k+1$ distinct curves that converge to $iy_0$. In other words, both curves in $\mathcal D_0$ emanating from  $iy_0$ with vertical slope bend into ${\mathcal H}_x$.  Furthermore, if $iy_0$ lies on $\mathcal C_0$, then by making $\delta$ smaller, if necessary, $ \mathcal C_0\cap B_{\delta}(iy_0)\cap {\mathcal H}_x$ consists of $k+2$ distinct curves that converge to $iy_0$.
\end{lemma}
\begin{proof}
Suppose not, and let $\gamma(s):[0,\epsilon_0)$ parametrize a curve on $\mathcal D_0$ that satisfies $\gamma(0)=iy_0$, $\gamma'(0)=\frac{\partial}{\partial y}$, yet $\gamma$ curves away from  ${\mathcal H}_x$. In other words, there exists a a $\delta>0$ so that the intersection  $\gamma\cap B_{\delta}(iy_0)\cap{\mathcal H}_x=\varnothing.$ Now, consider the one parameter family of level sets given by: $$\mathcal D_{0,t}:=\{z\,|\, {\rm Re}(e^{it}w'(z))=0\}.$$
When $t=0$, this is just the level set $\mathcal D_{0}$, and for small   $t$  we get a deformation of $\mathcal D_0$, which fixes all the zeros of $w'(z)$ and causes the tangent cones to rotate clockwise (as $t$ increases).

Now, for small enough $\delta$, if we restrict to the ball $B_{iy_0}(\delta)$, then $\mathcal D_{0,t}$ is a small deformtation  of the tangent cone $T^{iy_0}_{\mathcal D_{0,t}}.$ In particular this implies that $\mathcal D_{0,t}\cap\overline{B_{iy_0}(\delta)}$ consists of the fixed point $iy_0$, $2n$ points around the boundary $\partial B_{iy_0}(\delta)$, and $2n$ curves connecting $iy_0$ to the boundary points. For $t=0$ one of these $2n$ curves is our curve $\gamma(s)$, and by assumption it connects to a point $\zeta$ on $\partial B_{iy_0}(\delta)$ with ${\rm Re}(\zeta)<0$. Let $\gamma_t(s)$ and $\zeta_t$ be the deformation of these objects in $t$. 

For small enough $t$ we still have ${\rm Re}(\zeta_t)<0$. However, looking at how the tangent cone $T^{iy_0}_{\mathcal D_{0,t}}$ rotates, we see $\gamma_{t}'(s)=\kappa_1 \frac{\partial}{\partial x}+\kappa_2\frac{\partial}{\partial y}$, with $\kappa_1>0$ when $t>0$. As a result $\gamma_t(s)$ must enter ${\mathcal H}_x$ and achieve  vertical slope before heading back to $\zeta_t$. Yet $P(z)=e^{it}w'(z)$ is a polynomial with no roots in ${\mathcal H}_x$, so by Lemma \ref{vertical} the level set $\mathcal D_{0,t}$ does not achieve vertical slope in ${\mathcal H}_x$, a contradiction.

In the case that $iy_0$ lies on $\mathcal C_0$, we can make the same argument for curve $\Gamma_0(s)$  with vertical slope at $iy_0$ which heads away from ${\mathcal H}_x$. Looking at the deformation $$\mathcal C_{0,t}:=\{z\,|\, {\rm Im}(e^{it}w(z))=0\}$$
we again see for small $t>0$, the deformed curve $\Gamma_t(s)$ achieves vertical slope at a point $P$ in ${\mathcal H}_x$ and leaves along the $y$-axis. In this case Lemma \ref{vertical} does not apply, yet we can still arrive at a contradiction. Simply note that since $\Gamma_t(s)$ has vertical slope at $P$, there must be a curve from $\mathcal D_{0,t}$ through $P$. This implies there is a region $A_{k,t}$ which contains a curve $\Gamma_t(s)$ that lies on $\mathcal C_{0,t}$ and leaves  $A_{k,t}$ at both $iy_0$ and $P$. At the same time $A_{k,t}$ must also contain an asymptotic ray from $\mathcal C_{0,t}$, which gives two distinct curves form $\mathcal C_{0,t}$ in $A_{k,t}$, violating Proposition \ref{onecurveregion}.

\end{proof}

The previous two lemmas demonstrate that  if  $\mathcal C_0$ passes through generic critical point $iy_0$ of order $k$,  there will be $k+1$ components of $\mathcal C_0\cap{\mathcal H}_x$ that converge to $iy_0$, none of which have vertical slope on the $y$-axis. If $\mathcal C_0$ passes through  a non-generic critical point $iy_0$ of order $k$,  there will be $k+2$  components of $\mathcal C_0\cap{\mathcal H}_x$ that converge to $iy_0$, and the top and bottom components will have vertical slope.

\section{Stability and the Cauchy index}
\label{stabsection}

In the previous section with outlined the basic geometry of our setup. Working on ${\mathcal H}_x$, we saw that $\mathcal D_0\cap {\mathcal H}_x$ consists of $n-1$ distinct curves $\gamma_1,...,\gamma_{n-1}$ that do not self intersect nor achieve vertical slope (we note there are $n-2$ curves in the non-generic case, but for simplicity we follow the convention that we will use the generic case numbers when there are no differences other than numbering). We then partitioned $ {\mathcal H}_x\backslash \mathcal D_0$ into regions $A_1,..., A_n$, and proved Proposition \ref{onecurveregion}, which states that  for each $k$, the intersection $A_k\cap \mathcal C_0$ consists of exactly one distinct curve $\Gamma_k$.

Furthermore the curves $\Gamma_1,...,\Gamma_n$  are asymptotic to $n$ rays with angles $\frac\pi n$ between. Heading back towards the $y-$axis from the asymptotic ends, the curves can have one of   two main behaviors. First, they can stay inside a region $A_k$ all the way to the $y-$axis, never achieving vertical slope except possibly at the $y-$axis. In this case they will either end at  a critical point, continue into the left half-plane, or achieve vertical slope at the $y-$axis (not at a critical point) and head back into ${\mathcal H}_x$. Alternatively,  the curve  can exit $A_k$ before it reaches the $y$-axis, and achieve vertical slope on the boundary,  either heading into  $ A_{k+1}$ above (becoming $\Gamma_{k+1}$) or descending into  $A_{k-1}$ below (becoming $\Gamma_{k-1})$, and then turning back  toward the asymptotic ray in that region. In the non-generic case it is possible for a curve $\Gamma_1$ or $\Gamma_{n-1}$ to be asymptotic to both the $y-$axis and the ray with angle $\frac\pi n$ off of the $y$-axis. 

In the introduction  we stated Definition \ref{stabledef}, which characterizes  stability for the boundary data $\vartheta=\{z_1,z_2\}$ based on  existence of a function $f$ lying on $\mathcal C_0$ and satisfying the boundary conditions. Next we relate this definition to where the points in $\vartheta$ lie in relation to $A_1,..., A_n$. Because the location of these sets depends on knowing analytic details about $\mathcal C_0$ and $\mathcal D_0$, the following proposition does not really address Question \ref{mainquestion}. However, it provides a convenient framework to talk about stability which will be useful when we discuss the more algebraic conditions that follow. Recall that semi-stability includes both the stable and strictly semistable case.
Also, as above we assume $0\leq$ Re$(z_1)=a<b=$ Re$(z_2)$.

\begin{proposition}
\label{stabilityanalytic}
The boundary data $\vartheta$ is  semistable if and only if  $z_1\in\overline{A}_k$ and $z_2\in A_k$  for some $k$. Furthermore, in the semistable case:

\begin{enumerate}[label=(\roman*)] 
\item If $z_1$ is not a critical point of $w(z)$, then  $\vartheta $ is stable if and only if $z_1\notin\mathcal D_0$. 
\item If $z_1$ is a generic critical point of $w(z)$, then $\vartheta$ is always stable. 

\item If $z_1$ is a non-generic critical point and $z_1\in (\overline {A_{k-1}}\cap \overline{A}_k\cap\overline {A_{k+1}})$, with $z_2\in A_k$, then $\vartheta$ is stable. Otherwise $\vartheta$ is  strictly semistable. 
	\end{enumerate}

\end{proposition}

\begin{proof} We begin with the case that $z_1$ and $z_2$ do not lie in the closure of a single region $A_k$. Propostion \ref{onecurveregion} implies a component of  $\mathcal C_0$ can cross a curve in $\mathcal D_0\cap{\mathcal H}_x$ only if it achieves vertical slope at the crossing point $\zeta$, turning towards  the asymptotic ray in the regions above and below $\zeta$.   Thus if $z_1$ and $z_2$ do not lie on adjacent regions, they can not be on the same component $\mathcal C_0$, and so no graph exists between them.   If $z_1$ and $z_2$  lie in adjacent regions (say $A_k$ and $A_{k+1}$) and on the same component of $\mathcal C_0$, by the above ${\rm Re}(\zeta)<{\rm Re}(z_1)<{\rm Re}(z_2)$, so if $\mathcal C_0$ starts at $z_1$, and then passes through $\zeta$ before connecting to $z_2$, it will not stay graphical, so no graph can exists connecting $z_1$ and $z_2$. So $\vartheta$ is unstable.

If $z_2$ lies on $\partial A_k$ for some $k$, then again the level curve $\mathcal C_0$ through $z_2$ must turn towards  the asymptotic ray in the regions above and below $z_2$. So no component of $\mathcal C_0$ can connect $z_2$ to $z_1$ if ${\rm Re}(z_1)<{\rm Re}(z_2)$. Thus in this case $\vartheta$ is again unstable. This proves that if $\vartheta$ is  semistable, then   $z_1\in\overline{A}_k$ and $z_2\in A_k$  for some $k$.

Now, suppose $z_1\in \overline{A_k}$, $z_2\in A_k$, and   $z_1$ is not a critical point. If $z_1\notin\mathcal D_0$, and $a>0$, then $z_1$ is in the interior of $A_k$. Because there exists exactly one curve from $\mathcal C_0$ in $A_k$, and this curve can not achieve vertical slope within this region, we can find a smooth function $f$ lying on $\mathcal C_0$ with the correct boundary values.  If  $z_1\notin\mathcal D_0$, and $a=0$, then the curve on $\mathcal C_0$ through $z_1$ does not achieve vertical slope and must extend into $A_k$ and thus  connect to $z_2$, and we can again find a smooth function $f$ on $\mathcal C_0$ with the correct boundary values. Thus $\vartheta$ is stable if  $z_1\notin\mathcal D_0$.

Now, suppose $z_1\in \mathcal D_0$ and $z_1$ is not a critical point of $w$. Then the level curve $\mathcal C_0$ through $z_1$ will have vertical slope at $z_1$ before heading into $A_k$ towards $z_2$ (in the case that $z_1$ is on the $y$-axis we need Lemma \ref{curveinlemma} to guarantee $\mathcal C_0$ heads into   ${\mathcal H}_x$ towards $z_2$). Thus there exists a continuous function $f$ lying on $\mathcal C_0$ with the correct boundary values,   which is not $C^1$ at $z_1$. So $\vartheta$ is strictly semistable and not stable. This completes case $(i)$.

If  $z_1\in \overline{A_k}$, $z_2\in A_k$, and $z_1$ is a generic critical point of $w$, then by our above discussion of the tangent cone at $z_1$, the level set $\mathcal C_0$ emanates from $z_1$ into every region $A_\ell$ satisfying $z_1\in\overline{A_\ell}$, and the slope of the level curve remains bounded at $z_1.$ By Propostion \ref{onecurveregion} in each region the level curve  extends graphically to an asymptotic ray. Thus there exists a smooth function  $f$ lying on $\mathcal C_0$ with the correct boundary values. This completes case $(ii)$.

Finally, suppose $z_1\in \overline{A_k}$, $z_2\in A_k$, and let $z_1$ be a non-generic critical point of order $r$. By our discussion following Lemma \ref{curveinlemma} the level set $\mathcal C_0$ again emanates from $z_1$ into every region $A_\ell$ satisfying $z_1\in\overline{A_\ell}$, yielding $r+2$ components of $\mathcal C_0$, which we label as $\Gamma_{m},...,\Gamma_{m+r+1}$. The slope of $\Gamma_m$ and $\Gamma_{m+r+1}$  becomes unbounded  at $z_1$, while  for the other curves it remains bounded. If a point $z_2$ lies on  $\Gamma_m$ or $\Gamma_{m+r+1}$, then the region either above or below the one containing $z_2$ does not have $z_1$ in its closure, and in this case we can find a $C^0$ function $f$ lying on $\mathcal C_0$ which is not $C^1$ at $z_1$, and $\vartheta$ is strictly semistable. Otherwise  $z_2$ lies on  a curve in $\{\Gamma_{m+1},...,\Gamma_{m+r}\}$, and a smooth function $f$ lying on $\mathcal C_0$ connecting $z_1$ to $z_2$ exists, and $\vartheta$ is stable. This completes case $(iii)$.

  \end{proof}

With the above proposition we have  reduced our existence problem to determining which regions $A_1,..., A_n$ (and their closures) our points $z_1$ and $z_2$ lie in. The next step is to count these regions using the Cauchy index. We use the limit definition of the Cauchy index from \cite{E}. Let $h(y)$ be a rational function of a single real variable $y$.  For  a point $s\in \mathbb R$ define: 
\be
{\rm ind}_s^\epsilon(h)=\begin{cases}
+\frac12{\phantom{XXX}}&{\rm if}\,\lim_{y\rightarrow s^\epsilon} h(y)=+\infty\\
-\frac12&{\rm if}\,\lim_{y\rightarrow s^\epsilon} h(y)=-\infty\\
0&{\rm otherwise}
\end{cases}\nonumber
\ee
where $\epsilon$ is either $+$ or $-$.

\begin{definition}
The Cauchy index of $h $ at $s$ is defined as $${\rm ind}_s(h)={\rm ind}_s^+(h)-{\rm ind}_s^-(h).$$
\end{definition}
Intuitively, the Cauchy index  ${\rm ind}_s(h)= +1$ if $h$ jumps from $-\infty$ to $+\infty$, ${\rm ind}_s(h)= -1$ if $h$ jumps from $+\infty$ to $-\infty,$ and ${\rm ind}_s(h)= 0$ otherwise. Now, for any closed interval $[c,d]\subseteq \mathbb R$ (where $c$ and $d$ are allowed to be $\pm\infty$), the above definition can be extended to the interval as follows:

\begin{definition}
The Cauchy index of $h $ on an interval $[c,d]$ is given by  $${\rm ind}_c^d(h)={\rm ind}_c^+(h)+\sum_{x\in(c,d)}{\rm ind}_x(h)-{\rm ind}_d^-(h).$$
\end{definition}
Note that if   $h$ is non singular on the boundary of $[c,d]$ then ${\rm ind}_c^d(h)$ is an integer. We remark that in addition to looking at the discontinuities of $h$, the Cauchy index can also be calculated from the polynomials that make up $h$ using iterated Euclidean division. This is the content of Sturm's theorem, and we refer the reader to \cite{E} for details.

To see how the Cauchy index is a useful tool, we first consider the simple case where $z_1, z_2\in {\mathcal H}_x$, that is $0<a<b$. For any $\ell\in\mathbb R$ define the following rational function 
\be
R_\ell(y)=\frac{{\rm Im}(w'(\ell+iy))}{{\rm Re}(w'(\ell+iy))},\nonumber
\ee
which  is singular at $y_0$  if $\ell+iy_0$ lies on $\mathcal D_0$ and $\ell>0$. Thus  each time the vertical line ${\rm Re}(z)=\ell$ crosses a curve $\gamma_k$ in $\mathcal D_0$ it contributes  to the Cauchy index. This allows us to easily compute which  region $A_k$ a given point in ${\mathcal H}_x$ lies in, as follows:

\begin{lemma}
\label{labelinglemma}
Fix a point  $z_0=x_0+iy_0\in {\mathcal H}_x$, and consider the  function $N:{\mathcal H}_x\longrightarrow \frac12\mathbb Z$ defined via
\be
\label{countingfunction1}
N(z_0)=|{\rm ind}_{-\infty}^{y_0}(R_{x_0})|+1.
\ee
If $N(z_0)$ is an integer $k$, than $z_0\in A_k$. If $N(z_0)$ is not an integer, it is instead of the form $k-\frac12$, and $z_0$ lies on the boundary curve $\gamma_k$. 
\end{lemma}
\begin{proof}
To prove this result, we need to show that index of  every discontinuity   in $(-\infty, y_0]$ has the same sign, so there is no cancelation when computing the total index over the interval. As above, let $\gamma_1,...,\gamma_{n-1}$ be the distinct curves making up $\mathcal D_0\cap{\mathcal H}_x$. For $|z|$ large the curves $\gamma_1,...,\gamma_{n-1}$ are asymptotic to the tangent cone at infinity  $T^\infty_{\mathcal D_0}=\{z\,|\,{\rm Re}(\kappa (z-\hat z)^{n-1})=0\}$, where $\hat z$ is the average of the zeros of $w(z)$.  If in addition we consider the set $\{z\,|\,{\rm Im}(\kappa (z-\hat z)^{n-1})=0\}$, we see it consists of rays that alternate with those from $T^\infty_{\mathcal D_0}$, and thus there are  $n-2$ curves $\ti\gamma_1,...,\ti\gamma_{n-2}$ that lie on the level set $\{z\,|\,{\rm Im}(w'(z))=0\}$ that alternate with $\gamma_1,...,\gamma_{n-1}$. Specifically, for any $k\in\{1,...,n-2\}$, $\ti\gamma_k$ lies above $\gamma_{k}$ and below $\gamma_{k+1}$, and the curves never intersect because $w'(z)$ has no critical points in ${\mathcal H}_x$.

Without loss of generality, assume ${\rm Re}(w'(z))>0$ in $A_1$, which impies ${\rm Re}(w'(z))>0$ in $A_k$ if and only if $k$ is odd. Notice that along  $\gamma_1$, the sign of ${\rm Im}(w'(z))$ does not change. Looking at the tangent cone at infinity we directly see in this case that ${\rm Im}(w'(z))>0$ along $\gamma_1$. Now, the curve $\ti \gamma_1$ (where ${\rm Im}(w'(z))=0$) lies between $\gamma_1$ and $\gamma_2$, and so along $\gamma_2$ we know ${\rm Im}(w'(z))<0$. This alternating condition is preserved, so in this case we  have $ {\rm Im}(w'(z))>0$ on $\gamma_k$ if and only if $k$ is odd. Thus along the vertical line ${\rm Re}(z)=x_0$, starting from $y$ close to $-\infty$ and traveling upwards, each time $y$ crosses a curve $\gamma_k$ the quotient $R_{x_0}(y)$ will always go from positive infinity to negative infinity as it passes the discontinuity. This consistency implies that each curve $\gamma_k$ contributes the same sign to the index ${\rm ind}_{-\infty}^{y_0}(R_{x_0})$, proving the lemma.
\end{proof}

This Lemma now gives a satisfactory answer Question \ref{mainquestion} in the simple case where $\vartheta=\{z_1, z_2\}\subset {\mathcal H}_x$. First, from the discussion following the statement of Proposition \ref{onecurveregion} on the shape $\mathcal C_0$ in ${\mathcal H}_x$, we know  in this case  that $z_1$  and $z_2$ can not both lie on the same boundary  curve $\gamma_k$. Now, by Proposition \ref{stabilityanalytic}, we see right away that $\vartheta$ is stable if and only if 
$$|N(z_1)-N(z_2)|=|{\rm ind}_{-\infty}^{p}(R_{a})-{\rm ind}_{-\infty}^{q}(R_{b})|=0,$$
$\vartheta$ is strictly semistable if and only if 
$$|N(z_1)-N(z_2)|=|{\rm ind}_{-\infty}^{p}(R_{a})-{\rm ind}_{-\infty}^{q}(R_{b})|=\frac12$$
and $\vartheta$ is unstable if and only if 
$$|N(z_1)-N(z_2)|=|{\rm ind}_{-\infty}^{p}(R_{a})-{\rm ind}_{-\infty}^{q}(R_{b})|\geq1.$$

When $a=0$ the situation may be more complicated, depending on if there are critical points on the $y$-axis. If there are no such critical points, then everything works exactly  as above. Otherwise these critical points need to be taken into account in order to extend the above counting function $N$ to the $y$-axis. First, we need the following result.

\begin{lemma}
Let $iy_0$ be a critical point of $w(z)$ on the $y$-axis. If $iy_0$ is a generic critical point, then   ${\rm ind}_{y_0}(R_0)=0$. If  $iy_0$ is a non-generic critical point, then  ${\rm ind}_{y_0}(R_0)=-1$.
\end{lemma}
\begin{proof}
First, suppose $iy_0$ is a  generic critical point of order $k$. Because the Cauchy index ${\rm ind}_{y_0}(R_0)$ only depends on the behavior of $w'(z)$ in an arbitrarily small ball around $iy_0$, it suffices to look at the tangent cone $T^{iy_0}_{\mathcal D_0}$. In particular this means for some fixed $\phi_0$ we should compute the Cauchy index at $0$ of the  rational function
\be
\frac{{\rm Im}(e^{i\phi_0}(iy)^k)}{{\rm Re}(e^{i\phi_0}(iy)^k)},\nonumber
\ee
which is equal to ${\rm tan}(\phi_0)$ if $k$ is even and $-{\rm cot}(\phi_0)$ if $k$ is odd. As discussed above, because we are in the  generic case, if $k$ is even $\phi_0$ can not be $\pm\frac\pi 2$ and if $k$ is odd $\phi_0$ can not be $0$ or $\pi$, ensuring that the above rational function is always a well defined constant. The Cauchy index of a constant is always $0$, completing   the  generic case.

In the non-generic case,  the rational function associated to the tangent cone is identically $\infty$. Thus if $iy_0$ is a non-generic critical point of order $k$, $R_0(y_0)=\infty$, and thus we can compute the index at this point by looking at the sign of $R_0$ on either side of the singularity.  Note there exists a small $\delta>0$ for which   $y_0$ is the only point in the interval   $(y_0-\delta, y_0+\delta)$ where either ${\rm Re}(w'(iy))$ or ${\rm Im}(w'(iy))$ vanishes. As we have seen, the intersection $\mathcal D_0\cap B_\delta(y_0)$ consists of $2k$ curves emanating from the $iy_0$ and intersecting $\partial B_\delta(y_0)$ at $2k$ points. The set  $B_\delta(y_0)\backslash \mathcal D_0$ then consists of $2k$ sectors, and the sign of ${\rm Re}(w'(z))$ alternates in adjacent sectors. Now, by Lemma \ref{curveinlemma}, the two curves $\mathcal D_0$ with vertical slope at $iy_0$ must curve into ${\mathcal H}_x$, so given the $2k$ points on $\partial B_\delta(y_0)$, $2k+1$ must lie in ${\mathcal H}_x$. This implies ${\rm Re}(w'(i(y_0-\delta)))$ and ${\rm Re}(w'(i(y_0+\delta)))$ have the same sign. 

To check the sign of ${\rm Im}(w'(i(y_0-\delta)))$ and ${\rm Im}(w'(i(y_0+\delta)))$, note that the non-generic assumption implies that the none of the curves making up $\{z\,|\, {\rm Im}(w'(z))=0\}$ have vertical slope at $iy_0$ (since these curves emanate from the critical point  at angles shifted $\frac\pi{2k}$ from   $T^{iy_0}_{\mathcal D_0}$). Thus by making $\delta$ small enough, the alternating sign condition gives ${\rm Im}(w'(i(y_0-\delta)))$ and ${\rm Im}(w'(i(y_0+\delta)))$ have opposite sign. Now, regardless of the sign of ${\rm Re}(w'(i(y_0-\delta)))$ (it can be either positive or negative), we can see explicitly that $R_0(y_0-\delta)>0$ and $R_0(y_0+\delta)<0$, and so ${\rm ind}_{y_0}(R_0)=-1.$

\end{proof}

As we saw in the proof of Lemma \ref{labelinglemma}, the contribution to the index of a non critical point $iy$ where ${\rm Re}(w'(iy))=0$ is also $-1$. Thus we can use the Cauchy index, along with the knowledge of where the critical points lie,  to determine which sets in  $\{\overline{A_1},...,\overline{A_n}\} $  contain a particular point $iy$ on the $y$-axis. 

Specifically, consider the function ${\rm crit}_w(y):\mathbb R \rightarrow \mathbb Z$ defied by ${\rm crit}_w(y)=k$ if $iy$ is a critical point of $w$ of order $k$, and ${\rm crit}_w(y)=0$ otherwise.  Extend the counting function $N$ defined in \eqref{countingfunction1} to the $y-$axis by setting
\be
\label{countingfunction2}
N(iy_0)=|{\rm ind}_{-\infty}^{y_0}|+ \sum_{y\in(-\infty, y_0]}{\rm crit}_w(y)+1.
\ee
In other words, starting at $-\infty$ and moving up the $y-$axis, $N$ starts off at $1$ and increases by one every time it passes over a point $iy$ which is not a critical point yet ${\rm Re}(w'(i y))=0$. $N$ increases by $k$ every time it passes over a generic critical point of order $k$, and it increases by $k+1$ every time it passes over a non-generic critical point of order $k$. 

The location of a point $iy$   on the $y$-axis can now be determined as follows. If $iy$ is not a critical point, and and $N(iy)$ is an integer $m$, then $iy$ is on the boundary of $A_m$ only. If $iy$ is not a critical point, and $N(iy)=m+\frac12$ for an integer m, then $iy$ lies on a curve in $\mathcal D_0$, and is on the boundary of both $A_m$ and $A_{m+1}$. If $iy$ is   a generic critical point of order $k$, then $iy$ is on the boundary of $A_{N(iy)-k}, A_{N(iy)-k+1},..., A_{N(iy)-1},$ and $A_{N(iy)}$. Finally if $iy$ is a non-generic critical point of order $k$, then $iy$ is on the boundary of $A_{N(iy)-\frac12-k}, A_{N(iy)-\frac12-k+1},..., A_{N(iy)-\frac12},$ and $A_{N(iy)+\frac12}.$ Notice that a generic critical point is on the boundary of $k+1$ regions while a non-generic critical point is on the boundary of $k+2$ regions. 

This gives a complete description of how to use $N$ to determine which regions a point in $\overline {{\mathcal H}_x}$ lies on (or is on the boundary of), which via Proposition \ref{stabilityanalytic}, gives an answer Question \ref{mainquestion} when $z_1 $ lies on the $y$-axis and  $z_2\in {\mathcal H}_x$. The above paragraph, along with the discussion following Lemma \ref{labelinglemma}, allows us to conclude Theorem \ref{Cauchystable}.

\section{Geodesics in the space of functions}
\label{functionspace}

We have now shown how to use the Cachy Index to provide a satisfactory answer to Question \ref{mainquestion}. Our next goal is to relate our notations of stability to infinite dimensional Geometric Invariant Theory. As a first step, we   define our space of functions, and describe its behavior.   We then introduce a metric  and write down the geodesic equation. The space of functions we consider arises from the Calabi Symmetry construction (see Section \ref{DHYMsection}). Since this construction is our main geometric application, and fits nicely into the GIT framework, we choose to work with it here, even if it seems overcomplicated at first.

First we consider  the following function space:

\begin{definition}
Let  $ M_{p,q}$  be the space of all real functions $g(\rho):\mathbb R\rightarrow\mathbb R$ such that the corresponding functions $\hat g_{-\infty},\hat g_{\infty}:(0,\infty)\rightarrow\mathbb R$ defined via
$$\hat g_{-\infty}(e^\rho)=g(\rho)-p\,\rho \qquad{\rm and}\qquad \hat g_{\infty}(e^{-\rho})=g(-\rho)+q\,\rho,$$
can both be extended  smoothly to $[0,\infty)$. This forces $g$ to have the  asymptotic behavior
 \be
 \label{limith}
\lim_{\rho\ra-\infty}\frac{dg}{d\rho}=p,\qquad\lim_{\rho\ra\infty}\frac{dg}{d\rho}=q. 
\ee
  Note that $TM_{p,q}=M_{0,0}$.
\end{definition}

We use the variable $r=e^\rho$ in what follows. First, we fix a background function which plays the role of a potential for a metric. Choose $\mu\in M_{a,b}$ such that  $\frac{d\mu}{d\rho}>0,  \frac{d^2\mu}{d\rho^2}>0$, and in addition assume
 \be
 \label{gassump}
\qquad\frac{d}{dr}\hat \mu_{-\infty}(0)>0\qquad {\rm and}\qquad\frac{d}{dr}\hat \mu_{ \infty}(0)>0.
 \ee
  For notational simplicity set 
 $$ x:=\frac{d\mu}{d\rho}\qquad{\rm and}\qquad {\sigma}:=\frac{dx}{d\rho}= \frac{d^2\mu}{d\rho^2}.$$
 Since $x$ is strictly increasing it can be viewed as a coordinate on $(a,b)$, and each $x$ value uniquely specifies a value of $\rho$. 
 
 Next, we see any function $g(\rho)\in M_{p,q}$ can be used to define a function $f:(a,b)\rightarrow\mathbb R$ via
 \be
 \label{fdef}
 f(x):=\frac{dg(x)}{d\rho}=\frac{dg}{dx}\frac{d x}{d\rho}=\frac{dg}{dx}\sigma.
 \ee
 By \eqref{limith} we see that $f(x)$ extends as  continuous function  to $[a,b]$ with boundary conditions $f(a)=p$, $f(b)=q$. In fact, we can prove all derivatives of $f$ extend to the boundary continuously.

\begin{proposition}
\label{smoothextend}
 $f$ extends to $[a,b]$ as a smooth function up to the boundary. 
 \end{proposition}

 \begin{proof}
We focus our attention to the boundary point $a$, as extension to  $b$ is similar. For this proof we will use the notation $g', \mu'$, and so forth, to denote a derivative with respect to the variable $\rho.$

First, recall that $\hat g_{-\infty}(r)$ and $\hat \mu_{-\infty}(r)$ extend to smoothly to $[0,\infty)$.  In fact, each function can be extended to a smooth function on $(-\epsilon,\infty)$. To see this, note that  extension to $0$ gives two sequences  $\{\hat g^{(k)}_{-\infty}(0)\}_{k=0}^\infty$, and $  \{\hat \mu^{(k)}_{-\infty}(0)\}_{k=0}^\infty.$ By Borel's theorem (which was actually first proved by Peano, see \cite{Besenyei}) each sequence can be seen as the power series of a smooth function at $0$,  which can be used to extend $\hat g_{-\infty}(r)$ and $\hat \mu_{-\infty}(r)$ to $(-\epsilon,0)$.

We can now apply Taylor's Theorem at $0$, and write for any $k>0$
$$  \hat g_{-\infty}(r)=c_0+c_1 r+\frac{c_2}2 r^2+\frac{c_3}{3!}r^3+\cdots +\frac{c_k}{k!}r^k+R_k(r) r^k,$$
  with $c_k=\hat g^{(k)}_{-\infty}(0)$ and $\lim_{r\rightarrow 0}R_k(r)=0$. Similarly we have
  $$  \hat \mu_{-\infty}(r)=b_0+b_1 r+\frac{b_2}2 r^2+\frac{b_3}{3!}r^3+\cdots +\frac{b_k}{k!}r^k+\tilde R_k(r) r^k,$$
  with $b_k=\hat \mu^{(k)}_{-\infty}(0)$. By \eqref{gassump} we  know that $b_1>0$.  Returning to the definition of $\hat g_{-\infty}(r)$, and taking the first derivative yields
  $$\frac{d\hat g_{-\infty}}{dr}=\frac{d}{dr}\left(g({\rm log}r)-p{\rm log}r\right)= \frac{g'-p}r.$$
  Plugging this into the derivative of the power series, multiplying by $r$, and using that $f=g'$, we see that $f$ as a function of $r$ can be expressed as
\be
\label{powerseriesf}
f(r)=p+c_1 r+\frac{2c_2}2 r^2+\frac{3c_3}{3!}r^3+\cdots +\frac{kc_k}{k!}r^k+R_k(r) r^k.\nonumber
\ee
Similarly the power series for $x$ as a function of $r$ is given by
\be
\label{powerseriesx}
x(r)=a+b_1 r+\frac{2b_2}2 r^2+\frac{3b_3}{3!}r^3+\cdots +\frac{kb_k}{k!}r^k+\tilde R_k(r) r^k.
\ee

Next we verify that the  derivatives of $f(x)$  stay bounded as $x$ approaches $a$. Beginning with the first derivative, we see
\be
\label{fderivr}
\frac{df}{dr}=\frac{df}{dx}\frac{dx}{dr}.
\ee
Taking the limit as $r\rightarrow0$ (which is equivalent to $x\rightarrow a$) gives
$$c_1=\lim_{x\rightarrow a} \frac{df}{dx} b_1,$$
and so $\lim_{x\rightarrow a} \frac{df}{dx}=\frac{c_1}{b_1}<\infty$, since $b_1>0$. For the second derivative of $f(x)$, we take the  derivative of \eqref{fderivr} in $r$ and arrive at
$$\frac{d^2f}{dr^2}=\frac{d^2f}{dx^2}\left(\frac{dx}{dr}\right)^2+\frac{df}{dx}\frac{d^2x}{dr^2}.$$
Again taking the limit as  $r\rightarrow0$ yields
$$2c_2=\lim_{x\rightarrow a}\frac{d^2f}{dx^2}b_1^2+\frac{c_1}{b_1}2b_2,$$
which implies   that $\lim_{x\rightarrow a}\frac{d^2f}{dx^2}=\frac{2c_2}{b_1^2}-\frac{2c_1b_2}{b_1^3}$. We can continue this process and take arbitrary derivatives of \eqref{fderivr} in $r$, and then solve for $\lim_{x\rightarrow a}\frac{d^nf}{dx^n}$ for any $n$, using that $b_1>0$ for each step. This verifies that $f$ extends smoothly to $[a,b]$.

  \end{proof}

 As above we denote the points $z_1=a+ip$, $z_2=b+iq$, and write this set of this data as $\vartheta=\{z_1, z_2\}.$ We now restrict to a smaller class of functions, on which it will be easier to define a metric. First, we need the following assumption, which plays a big role in this section and the section which follows. We assume the point $z_2$ lies on the interior of $A_k$ for some $k$, with ${\rm Re}(w'(z))>0$ in $A_k$. This assumption is quite reasonable, since if $z_2$ were on $\gamma_k$ for some $k$,  then Proposition \ref {stabilityanalytic} would imply $\vartheta$ is unstable. Also, once we assume $z_2\in A_k$, we can always multiply  $w$ by $-1$ to ensure ${\rm Re}(w'(z))>0$.

Now, fix a function $g_0\in M_{p,q}$ with corresponding $f_0$, as in \eqref{fdef}. For any $\phi\in M_{0,0}$, let $f_\phi:=f_0+\frac{d\phi}{dx}\sigma$, which is another function in $M_{p,q}$. Consider the set 
\be
\mathcal M_{\vartheta}:=\{\phi\in M_{0,0}\,|\, \frac{d}{dx}{\rm Re}\left(w\left(x+if_\phi \right)\right)>0\}.\nonumber
\ee
Here we use the subscript $\vartheta$ to emphasize that this set depends on the choice of $g\in M_{a,b}$ and   $f_0\in M_{p,q}$, and thus it depends on the initial points $z_1, z_2$. By a slight abuse of notation we will write both $\phi\in \mathcal M_{\vartheta}$ and $f_\phi\in \mathcal M_{\vartheta}$ for functions in this space.

Our motivation for this definition is given by the following lemma.

\begin{lemma}
\label{lemmaM}
Suppose $\vartheta$ is stable with $z_1,z_2\in A_k$ and  ${\rm Re}(w'(z))>0$ in $A_k$. Let $f_\phi$ lie on $\mathcal C_0$. Then $\phi\in\mathcal M_\vartheta$. 
\end{lemma}
\begin{proof}
Because $f_\phi$ lies on $\mathcal C_0$, the function ${\rm Im}\left(w\left(x+if_\phi \right)\right) $ is constant in $x$, and so $\frac{d}{dx}{\rm Im}\left(w\left(x+if_\phi \right)\right)=0$. Bringing in the derivative we see
\be
0= {\rm Im}\left(w'\left(x+if_\phi \right)(1+i\frac{df_\phi}{dx})\right)={\rm Im}\left(w'\left(x+if_\phi \right) \right)+\frac{df_\phi}{dx}{\rm Re}\left(w'\left(x+if_\phi \right) \right).\nonumber
\ee
The above equation now implies
\bea
\frac{d}{dx}{\rm Re}\left(w\left(x+if_\phi \right)\right)&=& {\rm Re}\left(w'\left(x+if_\phi \right)(1+i\frac{df_\phi}{dx})\right)\nonumber\\
&=&{\rm Re}\left(w'\left(x+if_\phi \right) \right)-\frac{df_\phi}{dx}{\rm Im}\left(w'\left(x+if_\phi \right) \right)\nonumber\\
&=&{\rm Re}\left(w'\left(x+if_\phi \right) \right)+\left(\frac{df_\phi}{dx}\right)^2{\rm Re}\left(w'\left(x+if_\phi \right) \right)>0.\nonumber
\eea

\end{proof}

Thus $\mathcal M_{\vartheta}$ is non-empty in the stable case. We will also see it is non-empty in the semistable case, and can be either empty or non-empty in the unstable case. For the remainder of this section, to ease notation we denote by $z_\phi=x+if_\phi$ the set of points on the graph of $f_\phi$. We also introduce the set   $\tilde {\mathcal D}_0:=\{z\,|\,{\rm Im}(w'(z))=0\}$, which does not intersect $\mathcal D_0$ in ${\mathcal H}_x$.  Looking at the tangent cone at infinity, the curves in $\tilde{ \mathcal D_0}$ are asymptotic to rays that  alternate with those from $T^\infty_{\mathcal D_0}$. As a result each region $A_\ell$ must contain exactly one curve from the set $\tilde{ \mathcal D_0}$, which we denote by $\tilde \gamma_\ell$. The only   exception is $A_1$ and $A_n$, which may or may not contain a curve from $\tilde{ \mathcal D_0}$.

\begin{proposition}
\label{semistable}
Suppose $\vartheta$ is semistable. Then $\mathcal M_{\vartheta}$ is non-empty.
\end{proposition}

\begin{proof}
Fix a region $A_k$ where ${\rm Re}(w'(z))>0$.  Let $z_1$ lie on the boundary of $A_k$, and  $z_2\in A_k$.  By definition, any $f_\phi\in \mathcal M_\vartheta$ satisfies
\be
\label{inH}
\frac{d}{dx}{\rm Re} (w (z_\phi ) )= {\rm Re}  (w' (z_\phi  )  )-\frac{df_\phi}{dx}{\rm Im}(w'(z_\phi  ))>0.
\ee
Writing $w(z)=u +i v $, then $w'(z)=v_y+i v_x=u_x-i u_y$, and we see \eqref{inH} is equivalent to the condition that $\nabla u\cdot\vec{T}>0,$ where $\vec{T}:=\frac{\partial}{\partial x}+\frac{df_\phi}{dx}\frac{\partial}{\partial y}$ is the tangent vector to the graph of $f_\phi$.

In this case we cannot simply use the function on $\mathcal C_0$ connecting $z_1$ to $z_2$, since this function has vertical slope at $z_1$ and thus is not in $M_{p,q}$. Instead we argue as follows. Parametrize the curve   $\Gamma_k:[t_1,t_2]\rightarrow \mathcal C_0$ so $\Gamma_k(t_1)=z_1$ and $\Gamma_k(t_2)=z_2$. Without loss of generality assume $\Gamma_k'(t_1)=\frac{\partial}{\partial y}$ (otherwise we can just preform a mirror argument). The discussion following Proposition  \ref{onecurveregion} (or Lemma \ref{curveinlemma}) gives that $\Gamma_k$ stays to the right of the vertical line $x=a.$

Now, since the curve $\ti\gamma_k$ does not intersect $\gamma_k$, we can choose a ball $B_\delta(z_1)$ so that $B_\delta(z_1)\cap\ti\gamma_k=\varnothing.$ Choose $t^*$ small so $\Gamma_k([t_1,t^*])$ lies in $B_\delta(z_1)$. Now, since level sets of $u$ are always perpendicular to level sets of $v$, we know that $\nabla u$ points in the direction of $\Gamma_k'$ or $-\Gamma_k'$.  Lemma \ref{lemmaM} and \eqref{inH} imply that $\nabla u$ points in the direction of $\Gamma_k'$ and since $\Gamma'(t_1)=\frac{\partial}{\partial y}$ we see $u_y(z_1)>0$.   But since $\ti \gamma_k$ corresponds to the set where $\{u_y=0\}$, and this set does not intersect $B_\delta(z_1)$, it follows that $u_y>0$  in the whole ball. Furthermore, since ${\rm Re}(w'(z))>0$ in $A_k$, we have $u_x>0$ in $A_k\cap B_\delta(z_1)$.

Create a segment $L$ connecting $z_1$ to $\Gamma(t^*)$. Since $\Gamma(t^*)$ is to the right of the line ${\rm Re}(z)=a,$ we see right away that the tangent vector $\vec T$ to $L$  will have positive dot product with $\nabla u$. Next define a function $\hat f$ by first following $L$ from $z_1$ to $\Gamma_k(t^*)$, and then following $\Gamma_k(t)$ from  $\Gamma_k(t^*)$ to $\Gamma_k(t_2)=z_2$. We can now construct $f_\phi\in M_{p,q}$ by smoothing out $\hat f$ at the corner. Because the tangent vector to our segment $L$, and the tangent vector to $\Gamma_k(t)$, both have positive dot product with $\nabla u$ at this point, the smoothing $f_\phi$ can be chosen in such a way that preserves this positivity, and thus $f_\phi\in\mathcal M_\vartheta.$

\end{proof}

We now prove some results to help us get a better sense of what functions in  $\mathcal M_{\vartheta}$ look like.

\begin{lemma}
\label{crossinglemma1}
A function $f_\phi\in\mathcal M_{\vartheta}$ can not cross the curve $\tilde \gamma_\ell$ in a set $A_\ell$ where ${\rm Re}(w'(z))<0$.
\end{lemma}
\begin{proof}
By definition, any $f_\phi\in \mathcal M_\vartheta$ satisfies
\be
\frac{d}{dx}{\rm Re} (w (z_\phi ) )= {\rm Re}  (w' (z_\phi  )  )-\frac{df_\phi}{dx}{\rm Im}(w'(z_\phi  ))>0.\nonumber
\ee
At any point on $\ti \gamma_{\ell}$  we have ${\rm Im}(w'(z_\phi  ))=0$. If in addition ${\rm Re}(w'(z))<0$, then any curve $f_\phi$ that crosses $\ti \gamma_{\ell}$ satisfies $\frac{d}{dx}{\rm Re} (w (z_\phi ) )<0$,  so $f_\phi$ can not be in $\mathcal M_\vartheta$.
\end{proof}

The following lemma will also be quite useful.
\begin{lemma}
\label{crossinglemma2}
Let $\vartheta$ be simi-stable, with $z_1, z_2\in\overline{A_k}$, and ${\rm Re}(w'(z))>0$ in $A_k$. Then any $f_\phi\in \mathcal M_\vartheta$ can not intersect $\mathcal C_0$ outside of $\overline{A_k}$.
\end{lemma}
\begin{proof}
Although the boundary values of $f_\phi$ lie in $\overline{A_k}$, $f_\phi$ may  leave this set. However, by Lemma \ref{crossinglemma1}, $f_\phi$ can not cross $\tilde \gamma_{k-1}$ or $\tilde \gamma_{k+1}$, and so if $f_\phi$ intersects another curve from $\mathcal C_0$ outside of $A_k$, it must occur in either $A_{k-1}$ or $A_{k+1}$. Without loss of generality, assume $f_\phi$ intersects the curve $\Gamma_{k+1}$ in $A_{k+1}$. If this intersection occurs at one point, then $f_\phi$ must be tangent to $\Gamma_{k+1}$, and thus at this point $\frac{d}{dx}{\rm Im}(w(z_\phi))=0$. As in the proof of Lemma \ref{lemmaM},  it follows that
$$\frac{d}{dx}{\rm Re}(w ( z_\phi))={\rm Re} (w'(z_\phi))+\left(\frac{df_\phi}{dx}\right)^2{\rm Re} (w'(z_\phi) ).$$
However in $A_{k+1}$ we have ${\rm Re} (w'(z_\phi) )<0$, which implies $\frac{d}{dx}{\rm Re}( ( z_\phi))<0$, a contradiction.

Now suppose that $f_\phi$ intersects $\Gamma_{k+1}$ in $A_{k+1}$ at more than one point. Choose two intersection points and label them as $\zeta_1$ and $\zeta_2$.  Then because ${\rm Im}(w(\zeta_1))={\rm Im}(w(\zeta_2))=0$, the function $x\mapsto {\rm Im}(w(z_\phi))$ must have a critical point for some $x\in[{\rm Re}(\zeta_1),{\rm Re}(\zeta_2)]$. At this point  $\frac{d}{dx}{\rm Im}(w(z_\phi))=0$, and as above we arrive at a contradiction. Thus no such intersection can occur. 

\end{proof}
 
We now turn to the behavior of $\mathcal M_\vartheta$ in the unstable case.

\begin{proposition}
\label{unstable}
Suppose $\vartheta$ is unstable, and  $z_2\in A_k$ with ${\rm Re}(w'(z))>0$. If $z_1$ lies in $A_{k-1}$ or $ A_{k+1}$, $\mathcal M_{\vartheta}$ may be empty or non-empty. If $z_1$ lies outside of $A_{k-1}$ or $A_{k+1}$, then  $\mathcal M_{\vartheta}$ is empty. 
\end{proposition}

\begin{proof}
First, suppose $z_1$ lies outside of $A_{k-1}$ or $A_{k+1}$. Note that in both  $A_{k-1}$ and $A_{k+1}$ we have ${\rm Re}(w'(z))<0$, and furthermore any curve connecting $z_2$ to  $z_1$   must cross  $\ti\gamma_{k-1}$ or  $\ti\gamma_{k+1}$. By Lemma \ref{crossinglemma1} such a curve cannot be in $\mathcal M_{\vartheta}$, so $\mathcal M_{\vartheta}$ is empty in this case.

To see that $\mathcal M_{\vartheta}$ can be empty even when $z_1$ lies in an adjacent region $A_{k-1}$ or $A_{k+1}$, it is easy to find examples where portions of  the curve $\Gamma_{k+1}$ lie on the opposite side of $\ti\gamma_{k+1}$ as $z_2$. Thus  by choosing $z_1$ appropriately Lemma \ref{crossinglemma1} again implies $\mathcal M_{\vartheta}$ is empty.

We now show it is also possible for $\mathcal M_\vartheta$ to be nonempty in this case. In the discussion following Proposition \ref{onecurveregion}, we saw adjacent regions may   share a connected component of $\mathcal C_0$, and we assume this to be the case. Without loss of generality let $z_1\in A_{k-1}$. Our picture is now as follows. Starting at $z_1$, we can follow the curve $\Gamma_{k-1}$ in the negative $x$-direction, which heads towards $\gamma_{k-1}$ (the boundary between $A_{k-1}$ and $A_k$) and achieves vertical slope. At this point the curve  becomes $\Gamma_k$ and heads back towards the asymptotic ray in $A_k$ for $|z|$ large.

Since any curve in $\mathcal M_\vartheta$ can not cross $\ti\gamma_{k-1}$,   we assume $z_1$ lies between $\ti \gamma_{k-1}$ and $\gamma_{k-1}$. Furthermore, assume the intersection of the line $x=a$ with $\Gamma_k$ occurs at a point $\zeta$ which lies below $\ti\gamma_{k}$ (certainly this can be done by choosing $z_1$ close to $\gamma_{k-1}$). Following the same argument as in Proposition \ref{semistable}, we know $u_y>0$ in a small ball around $z_1$. This then implies $u_y>0$ between $\ti\gamma_{k-1}$ and $\ti\gamma_k$. Construct a linear function     $L$ connecting  $z_1$ to a point on $\Gamma_k$ in $A_k$,  which can be done so $L$ has large enough slope to ensure that the tangent vector $\vec T$ to $L$  will have positive dot product with $\nabla u$ over this segment. Then, just as in Proposition   \ref{semistable}, we can create a function $\hat f$ by first following $L$ from $z_1$ to $\Gamma_k$, and then following $\Gamma_k$ to $z_2$. This gives a function $f_\phi\in\mathcal M_\vartheta$  by smoothing out $\hat f$ at the corner. 
\end{proof}

We are now ready to introduce a metric on $\mathcal M_\vartheta$ and write down the geodesic equation.  For any point $\phi\in\mathcal M_\vartheta$, the tangents space satisfies $T_\phi\mathcal M_\vartheta=M_{0,0}$. For $\psi_1,\psi_2\in T_\phi\mathcal M_\vartheta$ (which again we can always regard as functions of $x$), we define the following metric
\be
\langle \psi_1,\psi_2\rangle=\int_a^b\psi_1\psi_2 \frac{d}{dx}{\rm Re}\left(w (z_\phi)\right)dx.\nonumber
\ee
From this metric we can write down the geodesic equation.
\begin{proposition}
A smooth  curve $\phi(t)\in\mathcal M_\vartheta$ is a geodesic if it solves 
$$0=\ddot \phi  \frac{d}{dx}{\rm Re}(w (z_\phi))+\left(\frac{d\dot\phi}{dx}\right)^2\sigma  {\rm Im} (w' (z_\phi) ).$$
\end{proposition}

\begin{proof}
Let $\phi(t)\in\mathcal M_\vartheta, t\in[0,1]$, be a curve with constant speed $1$. We define a surface $\phi(t,s):=\phi(t)+s\psi(t)$ in $ \mathcal M_\vartheta $ with $\psi(0)=0$, $\psi(1)=0$. For any fixed $s$ the arclength of $\phi(t,s)$ is given by
\be
L(s)=\int_0^1\sqrt{\int_a^b \dot \phi^2 \frac{d}{dx}{\rm Re}\left(w (z_\phi)\right)dx}dt.\nonumber
\ee
Taking the derivative in the $s$ direction at $0$
\bea
\label{geo1}
\frac{d}{ds}\Big|_{s=0}L&=&\int_0^1\frac1{ ||\dot\phi||}\int_a^b  \dot \phi\dot\psi \frac{d}{dx}{\rm Re}\left(w (z_\phi)\right) dxdt \\
&&+\int_0^1\frac1{2 ||\dot\phi||}\int_a^b   \dot \phi^2 \frac{d}{dx}\frac{d}{ds}{\rm Re}\left(w (z_\phi) \right)dxdt.\nonumber
\eea
Now, recall that $z_\phi=x+i(f(x)+\frac{d\phi}{dx}\sigma)$, and so 
$$\frac{d}{ds}{\rm Re}\left(w (z_\phi)\right)=-\frac{d\psi}{dx}\sigma{\rm Im}\left(w' (z_\phi)\right).$$
Thus after integration by parts in the $\frac{d}{dx}$ direction, and using the fact that $\frac{d\psi}{dx}$ vanishes at $a$ and $b$ by the definition of $M_{0,0}$, the second term on the right hand side of \eqref{geo1} becomes 
\be
\label{geo2}
\int_0^1\frac1{  ||\dot\phi||}\int_a^b   \dot \phi\frac{d\dot\phi}{dx} \frac{d\psi}{dx}\sigma{\rm Im}\left(w' (z_\phi)\right)dxdt.
\ee
Returning to the first term from \eqref{geo1}, we integrate by parts to move the $\frac{d}{dt}$ off of $\psi$. Using the fact that $\psi(0)=0$, $\psi(1)=0$, and $\phi(t)$ is constant speed, we see
\bea
\int_0^1\frac1{ ||\dot\phi||}\int_a^b  \dot \phi\dot\psi \frac{d}{dx}{\rm Re}\left(w (z_\phi)\right) dxdt&=&-\int_0^1\frac1{ ||\dot\phi||}\int_a^b  \ddot \phi \psi \frac{d}{dx}{\rm Re}\left(w (z_\phi)\right)  \nonumber\\
&&-\int_0^1\frac1{ ||\dot\phi||}\int_a^b  \dot \phi\ \psi \frac{d}{dx}\frac{d}{dt}{\rm Re}\left(w (z_\phi)\right).\nonumber 
\eea
The second term on the right above can be written as
\bea
&&\int_0^1\frac1{ ||\dot\phi||}\int_a^b  \dot \phi\ \psi \frac{d}{dx}\left(\frac{d\dot\phi}{dx}\sigma{\rm Im}\left(w' (z_\phi)\right)\right)dxdt  \nonumber\\
&&=-\int_0^1\frac1{ ||\dot\phi||}\int_a^b \left( \frac{d\dot \phi}{dx} \psi +\dot\phi\frac{d\psi}{dx} \right)\frac{d\dot\phi}{dx}\sigma{\rm Im}\left(w' (z_\phi)\right) dxdt.\nonumber 
\eea
Note the term above containing $\frac{d\psi}{dx}$ cancels with \eqref{geo2}. Putting everything together we now see that
$$\frac{d}{ds}\Big|_{s=0}L=-\int_0^1\frac1{ ||\dot\phi||}\int_a^b\psi \left( \ddot \phi  \frac{d}{dx}{\rm Re}(w (z_\phi))+\left(\frac{d\dot\phi}{dx}\right)^2\sigma  {\rm Im} (w' (z_\phi) ) \right) dxdt,$$
which completes the proof of the proposition. 
\end{proof}

\section{The Kempf-Ness functional and stability}
\label{KNsection}

In the previous section we introduced the space of functions $\mathcal M_\vartheta$, defined a metric on it, and wrote down the geodesic equation. Here we introduce the functional $\mathcal J:\mathcal M_\vartheta\rightarrow\mathbb R$, and demonstrate it is convex along geodesics. Because of this property, motivated by infinite dimensional geometric invariant theory, we refer to this functional as the $\mathcal J$-functional or the {\it Kempf-Ness functional}. The main goal of this section will be to relate properness of this functional to our notions of stability. 

\begin{definition}
The $\mathcal J$-functional is defined by its derivative along a path $\phi(t)\in\mathcal M_\vartheta$ as follows
\be
\frac{d}{dt}\mathcal J(\dot\phi)=-\int_a^b\dot\phi \frac{d}{dx}{\rm Im}(w(z_{\phi(t)}))=0.\nonumber
\ee
\end{definition}

In particular $d\mathcal J(\phi)=0$ at a point $\phi\in\mathcal M_\vartheta$ if and only if $f_\phi$ lies on a level curve of $\mathcal C_m$. Note this is a slightly different formulation from what was defined in the introduction, but we shall see the two are equivalent via integration by parts.

Previously we assumed that $z_2\in A_k$, and that ${\rm Re}(w'(z))>0$ in $A_k$, an assumption we   carry over to this section. This   implies critical points of  the  $\mathcal J$-functional  are absolute minimums in the stable case. As stated earlier, this  is a general assumption  in that we can always multiply our polynomial $w(z)$ by $-1$. Of course, if ${\rm Re}(w'(z))<0$, one could instead consider the space of functions  where $\frac{d}{dx}{\rm Re}(w(z_\phi))<0$, and then maximize $\mathcal J$ over this space.

\begin{proposition}
\label{convexgeodesics}
The $\mathcal J$-functional is convex along geodesics.
\end{proposition}
\begin{proof}
We compute directly
\bea
\frac {d^2}{dt^2}{\mathcal J}(\dot \phi)&=&-\int_a^b\left(\ddot \phi\, \frac{d}{dx}{\rm Im}(w(z_{\phi(t)}))+ \dot \phi\frac{d}{dx}\frac{d}{dt}  {\rm Im}(w(z_{\phi(t)}))\right)dx\nonumber\\
&=& \int_a^b\left(-\ddot \phi\, \frac{d}{dx}{\rm Im}(w(z_{\phi(t)}))- \dot \phi\frac{d}{dx} \left(\frac{d\dot\phi}{dx}\sigma{\rm Re}(w'(z_{\phi(t)}))\right)\right)dx.\nonumber
\eea
Integrating by parts on the second term, and using that $\frac{d\dot\phi}{dx}$ vanishes on the boundary, we see
\bea
\frac {d^2}{dt^2}{\mathcal J}(\dot \phi)&=& \int_a^b\left(-\ddot \phi\, \frac{d}{dx}{\rm Im}(w(z_{\phi(t)}))+ \left(\frac{d\dot\phi}{dx}\right)^2\sigma{\rm Re}(w'(z_{\phi(t)}))\right)dx.\nonumber
\eea
Plugging in the geodesic equation for $\ddot \phi$ yields 
\bea
\frac {d^2}{dt^2}{\mathcal J}(\dot \phi)&=& \int_a^b\left(\frac{\left(\frac{d\dot\phi}{dx}\right)^2\sigma  {\rm Im} (w' (z_\phi) )}{\frac{d}{dx}{\rm Re}(w (z_\phi))}\right) \frac{d}{dx}{\rm Im}(w(z_{\phi(t)}))dx\nonumber\\
&&+\int_a^b\left(\frac{d\dot\phi}{dx}\right)^2\sigma{\rm Re}(w'(z_{\phi(t)}))dx.\nonumber
\eea
In each of the terms on the right hand side above, we can factor out $\left(\frac{d\dot\phi}{dx}\right)^2\frac{\sigma  }{\frac{d}{dx}{\rm Re}(w (z_\phi))},$ which remains nonnegative as long as we are in    the space $\mathcal M_\vartheta$. Factoring out this expression we are left with 
\be
\label{geo3}
 {\rm Im} (w' (z_\phi) ) \frac{d}{dx}{\rm Im}(w(z_{\phi(t)}))+{\rm Re}(w'(z_{\phi(t)}))\frac{d}{dx}{\rm Re}(w (z_\phi)),
\ee
which we now check is nonnegative as well. Specifically
\bea
 \frac{d}{dx}{\rm Im}(w(z_{\phi(t)}))&=& {\rm Im}\left(w'(z_{\phi(t)})\left(1+i\frac{d}{dx}\left(\frac{d\phi}{dx}\sigma\right)\right)\right)\nonumber\\
 &=& {\rm Im}(w'(z_{\phi(t)}))+\frac{d}{dx}\left(\frac{d\phi}{dx}\sigma\right) {\rm Re} (w'(z_{\phi(t)}) ).\nonumber
 \eea
 Similarly
 $$ \frac{d}{dx}{\rm Re}(w(z_{\phi(t)}))=   {\rm Re}(w'(z_{\phi(t)}))-\frac{d}{dx}\left(\frac{d\phi}{dx}\sigma\right) {\rm Im} (w'(z_{\phi(t)}) ) .$$
 Thus, equation \eqref{geo3} can now be written as
  \bea
  &&{\rm Im} (w' (z_\phi) )^2 + {\rm Im} (w' (z_\phi) ) \frac{d}{dx}\left(\frac{d\phi}{dx}\sigma\right) {\rm Re} (w'(z_{\phi(t)}) )\nonumber\\
  &&+ {\rm Re}(w'(z_{\phi(t)}))^2- {\rm Re}(w'(z_{\phi(t)}))\frac{d}{dx}\left(\frac{d\phi}{dx}\sigma\right) {\rm Im} (w'(z_{\phi(t)}) )\geq0,\nonumber
 \eea
proving the $\mathcal J$-functional is convex along geodesics.

\end{proof}

In fact, the above argument works to show that the  $\mathcal J$-functional is convex along subgeodesics, i.e. solutions  to  $$\ddot \phi  \frac{d}{dx}{\rm Re}(w (z_\phi))+\left(\frac{d\dot\phi}{dx}\right)^2\sigma  {\rm Im} (w' (z_\phi) )\leq0.$$
We now relate the Kempf-Ness functional to our notions of stability.

\begin{proof}[Proof of Theorem \ref{functionalstability}]

To begin, first assume we are in the semistable case. By Proposition \ref{stabilityanalytic} we have $z_2\in A_k$ and  $z_1\in \overline{A_k}$, with ${\rm Re}(w'(z))>0$ in $A_k$.  Let $f$ be the $C^0[a,b]$ function lying on $\mathcal C_0$ connecting $z_1$ to $z_2$. Our  first goal is to show that in this case, any function in $\mathcal M_\vartheta$ can be deformed to $f$ in a way that decreases the $\mathcal J$-functional.

Fix a function $f_{\phi_0}\in \mathcal M_\vartheta$. For each $x\in[a,b]$, let $f_{\phi_t}(x)$ be the solution to the ODE
\be
\label{flow}
\dot f_{\phi_t}=-{\rm Im}(w(z_{\phi_t}))
\ee
with initial value $f_{\phi_0}(x)$. This defines a family of functions $f_{\phi_t}$ on $[a,b]$, which we think of as a zeroth-order flow deforming $f_{\phi_0}$ towards the curve $f$ where ${\rm Im}(w)=0.$
Note that ${\rm Im}(w(z_1))={\rm Im}(w(z_2))=0$, so $\dot f_{\phi_t}=0$ at the boundary points and thus the boundary condition is preserved.

From Lemma \ref{crossinglemma1} and Lemma \ref{crossinglemma2}, we have some control on  our initial function $f_{\phi_0}$. In particular, we know it is contained  in  $A_{k-1}\cup A_k\cup A_{k+1}$. Furthermore,  $f_{\phi_0}$ does not intersect $\mathcal C_0$ in either $A_{k-1}$ or $A_k$. The assumption that ${\rm Re}(w'(z))>0$ in $A_k$ implies $v_y>0$ in $A_k$, and since  $v=0$ along  the graph of $f$, we  see that $v(x,y)>0$ if $z$ lies above the graph of $f$ but below the next curve in $\mathcal C_0$. Similarly for points $z$ below the graph of $f$ but above the next curve in $\mathcal C_0$, we have $v(x,y)<0$.   Thus  for a given $x\in[a,b]$  one of 3 things can happen. If $f_{\phi_0}(x)>f(x)$, then  ${\rm Im}(w(x+ if_{\phi_0}(x)))>0,$ and so the flow \eqref{flow} decreases $f_{\phi_0}$  down towards $f$. If  $f_{\phi_0}(x )=f(x )$, we have ${\rm Im}(w(x+ if_{\phi_0}(x)))=0$, and so $f_{\phi_0}(x)$ is fixed. Finally, if $f_{\phi_0}(x)<f(x)$, then  ${\rm Im}(w(x+ if_{\phi_0}(x)))<0,$ and  the flow \eqref{flow} increases  $f_{\phi_0}$ up towards $f$.

\begin{lemma}
The family $f_{\phi_t}$   converges to $f$ in the $C^0$ norm as $t\rightarrow\infty$.
\end{lemma}
\begin{proof}
Fix $\epsilon>0$. Let $R$ denote the region bounded by the graph of $f_{\phi_0}$ and the graph of $f$, and let $R_\epsilon:=R\backslash B_\epsilon(f)$, where $B_\epsilon(f)$ is an epsilon tube around the graph of $f$ (in the Euclidean topology on $\mathbb C$). Because $f_{\phi_0}$ does not intersect $\mathcal C_0$ in $R_\epsilon$, we see $|{\rm Im}(w(z))|>0$ on $R_\epsilon$. Define $\eta_0:=\min_{R_\epsilon} |{\rm Im}(w(z))|>0$. Then, inside of $R_\epsilon$, when $f_{\phi_t}>f$ we have
 $\dot f_{\phi_t}=-{\rm Im}(w(z_{\phi_t}))<-\eta_0.  $
Similarly if $f_{\phi_t}<f, $ inside of $R_\epsilon$  we have $\dot f_{\phi_t}=-{\rm Im}(w(z_{\phi_t}))>\eta_0.  $ In either case, this bound shows that after a finite time $T$ the graph $f_{\phi_T}$ will lie inside $B_\epsilon(f)$, proving $C^0$ convergence.
\end{proof}

As for the rate of convergence, for a fixed $x\in(a,b)$, $f_{\phi_t}(x)$ converges exponentially fast to $f(x)$. To see this, note that since $x>a$ we have $f(x)\in A_{k}$, and thus ${\rm Re}(w'(x+ if (x))>0$. Now,  find an $\epsilon_x>0$ so that $B_{\epsilon_x}(x+if(x))\subset A_k$, and let $\delta_x>0$ be the infimum of ${\rm Re}(w'(z))$ over this ball. By the above convergence, after a finite time $T$, we have $x+if_{\phi_t}(x)\in B_{\epsilon_x}(x+if(x))$ for $t>T$. Then, in the case that $f_{\phi_t}>f$
$$\ddot f_{\phi_t}=-\frac{d}{dt}{\rm Im}(w(z_{\phi_t}))=-\dot f_{\phi_t}{\rm Re}(w(z_{\phi_t}))<-\delta_x\dot f_{\phi_t}.$$
If  $f_{\phi_t}<f$ we just switch the inequality in the above line.  In either case $|\dot f_{\phi_t}|$ is exponentially decaying in time  for $t>T$. Note that if $z_1\in\partial A_k$, then we may have ${\rm Re}(w'(z_1))=0$, in which case the above decay rate $\delta_x$ will degenerate as $x$ approaches $a$. On the other hand if ${\rm Re}(w'(z_1))>0,$ then the rate of exponential convergence will be uniform in $x$. 

Our next step is to show that  the $\mathcal J$-functional is nondecreasing along the flow \eqref{flow}.  Note that we do not claim that $f_{\phi_t}\in \mathcal M_\vartheta$ for $t>0$. Furthermore, if $\vartheta$ is strictly semistable than $f$ is explicitly not in $\mathcal M_\vartheta$, since it has infinite slope at $a$. However, the $\mathcal J$-functional is still defined, which is suitable for our purposes.   Recall that $f_{\phi_t}=f+\frac{d\phi_t}{dx}\sigma$, and so $\dot f_{\phi_t}=\frac{d\dot\phi_t}{dx}\sigma$. Integrating by parts gives:
\bea
\frac{d}{dt}\mathcal J(\dot\phi_t)&=&-\int_a^b\dot\phi_t \frac{d}{dx}{\rm Im}(w(z_{\phi_t}))=\int_a^b\frac{d}{dx}\dot\phi_t  {\rm Im}(w(z_{\phi_t}))\nonumber\\
\label{decreasing}&=&\int_a^b\frac{ \dot f_{\phi_t}}{\sigma}  {\rm Im}(w(z_{\phi_t}))=-\int_a^b\frac{ {\rm Im}(w(z_{\phi_t}))^2  }{\sigma}\leq 0. 
\eea
Note that the last inequality is strict unless ${\rm Im}(w(z_{\phi_t}))=0$, that is is unless $f_{\phi_t}$ lies on $\mathcal C_0$. By Proposition \ref{onecurveregion} there is exactly one curve in $\mathcal C_0$   connecting  $z_1$ to $z_2$. Thus, among all continuous functions with boundary values $z_1$ and $z_2$,  the  $\mathcal J$-functional has no critical points other then $f$.

We now have that $f_{\phi_t}\rightarrow f$ in $C^0$, and that the $\mathcal J$-functional is nondecreasing in time. Our next step is to argue that $\mathcal J(f_{\phi_t})$ converges to $\mathcal J(f)$. Looking at \eqref{decreasing} we see the derivative of the $\mathcal J$-functional  only depends on zeroth order data for $f_{\phi_t}$, so $C^0$ convergence of this family should be enough. Because $\sigma=0$ on the boundary points $a$ and $b$, we only need to check that the quotient $\frac{ {\rm Im}(w(z_{\phi_t}))  }{\sigma}$ stays a continuous function for all time. We focus our attention at the point $a$, as the argument at $b$ is similar. 

At the initial time, we have assumed $f_{\phi_0}\in\mathcal M_\vartheta$. As we saw in the proof of Propsosition \ref{semistable},   since level sets of $u$ are always perpendicular to level sets of $v$, we know that $\nabla u$ points in the direction of $\Gamma_k'$ or $-\Gamma_k'$.  Lemma \ref{lemmaM} and \eqref{inH} imply that $\nabla u$ points in the direction of $\Gamma_k'$. Because $f_{\phi_0}\in\mathcal M_\vartheta$, we know $\nabla u\cdot (\frac{\partial}{\partial x}+\frac{df_{\phi_0}}{dx}(a)\frac{\partial}{\partial y})>0.$  In particular this implies that at $z_1$, the graph of $f_{\phi_0}$ is not perpendicular to the level set $\Gamma_k$, and is instead bounded away from the perpendicular by some ray with angle $\theta_0>0$ off the perpendicular.  Thus ${\rm Im}(w(z_{\phi_0}))$ must go to zero at least linearly as $x$ goes to $a$, which gives that ${\rm Im}(w(z_{\phi_0}))/(x-a)$ is bounded.

It is helpful to return to the power series \eqref{powerseriesx}. This series is defined with respect to the variable $r=e^\rho$, and recall that $x\rightarrow a$ is equivalent to $r\rightarrow0$. Furthermore, recall that $\sigma=\frac{dx}{d\rho}=\frac{dx}{dr}\frac{dr}{d\rho}.$ Then $\frac{dr}{d\rho}=r$, which gives  $\lim_{r\rightarrow 0}\frac\sigma r=b_1,$
where $b_1$ arises from \eqref{powerseriesx}. Thus 
$$\lim_{r\rightarrow 0}\frac{x-a}\sigma=\lim_{r\rightarrow 0}\frac{(x-a)/r}{\sigma/r}=1.$$
From here we conclude
$$\lim_{x\rightarrow a}\frac{{\rm Im}(w(z_{\phi_0}))}{\sigma}=\lim_{x\rightarrow a}\frac{x-a}{\sigma}\,\frac{{\rm Im}(w(z_{\phi_0}))}{x-a}<C.$$
Now, because $f_{\phi_t}$ is bounded between $f_{\phi_0}$ and $f$ (which lies on $\Gamma_k$), it follows that for all $t$ the graph of $f_{\phi_t}$ must approach $z_1$ bounded by the same ray with angle $\theta_0>0$ off the perpendicular of $\Gamma_k'(z_1)$. In particular $\frac{ {\rm Im}(w(z_{\phi_t}))  }{\sigma}$ is bounded for all $t$. Thus the integrand of $\frac{d}{dt}\mathcal J(\dot\phi_t)$ stays bounded and continuous, and   since $f_{\phi_t}\rightarrow f$ in $C^0$, we conclude $\mathcal J(f_{\phi_t})\rightarrow\mathcal J(f).$

We have now show that, if $\vartheta$ is semistable, then for any $f_{\phi_0}\in \mathcal M_\vartheta$  
$$\mathcal J(f_{\phi_0})\geq \mathcal J(f)>-\infty,$$
and so the  $\mathcal J$-functional is bounded from below. If in addition $\vartheta$ is stable, then $f\in\mathcal M_\vartheta$ and this global minimum is attained. 

In fact, we can also see $f$ is a local minimum of ${\mathcal J}$   by taking the second variation. Let $f(t)=f+\frac{d\psi(t)}{dx}\sigma$, with $\psi(t)\in M_{0,0}$ and $\psi(0)=0$, be a variation of $f$. In the proof of Propostion \ref{convexgeodesics} we computed 
\bea
\frac {d^2}{dt^2}\Big|_{t=0}{\mathcal J}(\dot \psi)&=& \int_a^b\left(-\ddot \psi\, \frac{d}{dx}{\rm Im}(w(z_{\psi(0)}))+ \left(\frac{d\dot\psi}{dx}\right)^2\sigma{\rm Re}(w'(z_{\psi(0)}))\right)dx.\nonumber
\eea
Since $f_{\psi(0)}$ lies on $\mathcal C_0$, then $ \frac{d}{dx}{\rm Im}(w(z_{\psi(0)}))=0$, and the first term on the right hand side vanishes. Furthermore, we have assumed  ${\rm Re}(w'(z))>0$ in $A_k$. Thus if the variation $\psi(t)$ is non trivial, then 
\bea
\frac {d^2}{dt^2}\Big|_{t=0}{\mathcal J}(\dot \psi)>0.\nonumber
\eea

On the other hand, if  $\vartheta$ is strictly semistable, then   $f$ is not in $\mathcal M_\vartheta$. So while the  $\mathcal J$-functional is bounded from below, any function $f_{\phi_0}\in \mathcal M_\vartheta$ can always be deformed in a way that decreases $\mathcal J(\cdot)$, and a limit in $\mathcal M_\vartheta$ will never be attained. Thus the $\mathcal J$-functional is not proper in this case.

We now turn to the unstable case. If $\mathcal M_\vartheta$ is empty, by Lemma \ref{lemmaM} and Proposition \ref{semistable} we know $\vartheta$ is unstable. Thus to complete the proof of the theorem  we need to show  that if  $\mathcal M_\vartheta$ is non-empty in this case, then the $\mathcal J$-functional fails to be bounded from below. Choose a function $f_\phi\in\mathcal M_\vartheta$. 

By Proposition \ref{unstable}, $\mathcal M_\vartheta$ can only be non-empty if $z_1$ is either in $A_{k-1}$ or $A_{k+1}$. Assume without loss of generality that $z_1\in A_{k-1}$ (otherwise we can just preform a mirror argument). Note that ${\rm Re}(w'(z))<0$ $(u_x<0)$ in $A_{k-1}$. Parametrize $\Gamma_{k-1}$ so $\Gamma_{k-1}(0)=z_1$ and $|\Gamma_{k-1}(t)|\rightarrow\infty$ and $t\rightarrow\infty$. As we have seen $\nabla u$ points in the same direction as either $\Gamma'_{k-1}(t)$ or $-\Gamma'_{k-1}(t)$, and because $u_x<0$ in $A_{k-1}$ we see $\nabla u(z_1)$ points in the same direction as  $-\Gamma'_{k-1}(0)$.

Here we recall the defining property of $\mathcal M_\vartheta$:
\be
\label{inH2}
\frac{d}{dx}{\rm Re} (w (z_\phi ) )= {\rm Re}  (w' (z_\phi  )  )-\frac{df_\phi}{dx}{\rm Im}(w'(z_\phi  ))>0. 
\ee
When $f_\phi$ crosses $\gamma_{k-1}$  we have $ {\rm Re}  (w' (z_\phi  )  )=0$. Furthermore, $\gamma_{k-1}$ lies between $\ti \gamma_{k-1}$ and $\ti \gamma_{k}$, and in this region  ${\rm Im}(w'(z) )<0$  (as we saw in the proof of Proposition \ref{unstable}). Thus, as  $f_\phi$ crosses $\gamma_{k-1}$ we must have  $\frac{df_\phi}{dx}>0$. But   this implies $\frac{df_\phi}{dx}(a)>0$, since otherwise there would be a point in $A_{k-1}$ where $\frac{df_\phi}{dx}$ vanishes, which is impossible since  ${\rm Re}  (w' (z  )  )<0$ in $A_{k-1}$. Denote $\frac{df_\phi}{dx}(a)$ by $m_0.$ Since $\nabla u(z_1)$ points in the same direction as  $-\Gamma'_{k-1}(0)$, and $m_0>0$, we know $u_x(z_0)<0$, while $u_y(z_0)>0$. Since $f_\phi\in\mathcal M_\vartheta$ we know $\nabla u\cdot (\frac{\partial}{\partial x}+\frac{df_{\phi}}{dx}(a))>0$, or $m_0>\frac{u_x}{u_y}(z_1)$. Now,  choose a real number $m$ such that $m_0>m>\frac{u_x}{u_y}(z_1)$. By continuity of $u$, there exists  a $\delta>0$ so  that $(\frac{\partial}{\partial x}+m\frac{\partial}{\partial y})\cdot \nabla u(x,y)>0$ for every $z$ in  $B_\delta(z_1)$.

We may now construct a new continuous function $\hat f_0:[a,b]\rightarrow \mathbb R$ which is  $f_\phi$ outside of $B_\delta(z_1)$, and a linear function with slope $m$ inside. Let $f_0$ denote the smooth function given by smoothing out the corner on $\partial B_\delta(z_1)$.  Then $f_0$ satisfies condition \eqref{inH2}, but it fails to satisfy the boundary condition $f_0(a)=p$. Since the slope $m$ is smaller than $\frac{df_\phi}{dx}(a)$, we have $f_0(a)=p+\epsilon_0$ for some small $\epsilon_0>0$.

We are now ready to construct a family a functions in $\mathcal M_\vartheta$ for which the $\mathcal J$-functional approaches $-\infty$. For $t>t_0$ large enough, consider the family of linear functions 
$$L_t(x)=t(x-a)+p,$$
which intersects $f_0$ in $B_\delta(z_1)$. Furthermore, since $f_0$ is a linear function with positive slope in $B_\delta(z_1)$, the intersection of $L_t$ and $f_0$ occurs above the line $y=p+\epsilon_0$. Let $f_{\phi_t}$ be the family of functions that follows $L_t$  and then smoothly attaches to $f_0$, following $f_0$ to $z_2$. We can assume the smoothing takes place above the line $y=p+\epsilon_0$. By construction the family $f_{\phi_t}$ is in $\mathcal M_\vartheta$ for all $t>t_0$.

Next we compute the derivative of the $\mathcal J$-functional along this family. First, we note that $L_t$ lies above $\mathcal C_0$ in $B_\delta(z_1)$. Since $v_y<0$ in $A_{k-1}$ and $v=0$ along $\mathcal C_0$, it follows that $v(z_{\phi_t})={\rm Im}(w(z_{\phi_t}))<0$ in $B_\delta(z_1)$. Also since slope of the lines  $L_t$  increases with $t$, $\dot f_{\phi_t}\geq0$. We now have
\bea
\frac{d}{dt}\mathcal J(\dot\phi_t)=\int_a^b\frac{ \dot f_{\phi_t}}{\sigma}  {\rm Im}(w(z_{\phi_t}))&<&\int_a^{a+\epsilon_0/t}\frac{ \dot L_t}{\sigma}  {\rm Im}(w(z_{\phi_t}))\nonumber\\
&=&\int_a^{a+\epsilon_0/t}\frac{ x-a}{\sigma}  {\rm Im}(w(z_{\phi_t}))\nonumber\\
&<&\frac12\int_a^{a+\epsilon_0/t} {\rm Im}(w(z_{\phi_t})),\nonumber
\eea
where in the last line we used that $\lim_{r\rightarrow 0}\frac{x-a}\sigma=1$. Again $v=0$ along $\mathcal C_0$, and so $v(a+i(p+\epsilon_0))<0$. Making $t_0$ larger if necessary, by continuity of $v$ we can assume that along the portion of  $L_t$ that lies above $y=p+\epsilon_0/2$ we have $ {\rm Im}(w(z_{\phi_t}))<-\eta_0$ for some small constant $\eta_0>0$. Thus
$$\frac{d}{dt}\mathcal J(\dot\phi_t)<-\frac12\int_a^{a+\epsilon_0/t} \eta_0=-\frac12\frac{\eta_0\epsilon_0}{t}.$$
In particular
$$\lim_{t\rightarrow\infty}\mathcal J(\dot\phi_t)=\int_{t_0}^{\infty}\frac{d}{dt}\mathcal J(\dot\phi_t)dt<-\frac12 \int_{t_0}^\infty \frac{\eta_0\epsilon_0}{t}dt=-\infty,$$
and so the $\mathcal J$-functional is not bounded from below.

\end{proof}

\section{The deformed Hermitian-Yang-Mills equation on manifolds with Calabi Symmetry}
\label{DHYMsection}

We now turn to our main geometric application, constructing solutions to the deformed Hermitian-Yang-Mills equation on manifolds with Calabi Symmetry. We begin with the background geometry, and much of this discussion can also be found in \cite{Ca, D, FL, Song, SY}. We include it here for the reader's convenience. 

Let $E\rightarrow\mathbb P^m$ be the rank $r+1$ vector bundle over projective space associated to the locally free sheaf $\mathcal O_{\mathbb P^m}(-1)^{\oplus (r+1)}$. We consider the $m+r+1$ dimensional compact manifold
$$ X_{r,m}:=\mathbb P(\mathcal O_{\mathbb P^m}\oplus E),$$
which is a $\mathbb P^{r+1}$ bundle over $\mathbb P^m$, sometimes referred to as the projective completion of $E$. We denote by $D_\infty:=\mathbb P(E)$ the divisor at infinity in the fibers, which belongs to the linear system $|\mathcal O_{X_{r,m}}(1)|$. Let $D_H$ denote the pullback of the hyperplane divisor from $\mathbb P^m$.  By \cite{D} we know the vector space $N^1(X_{r,m},\mathbb R)$ (the space of numerically equivalent Cartier divisors), is spanned by $[D_H]$ and $[D_\infty]$. There is also a divisor class
$$[D_0]=[D_\infty]-[D_H]$$
which contains the quotient $\mathbb P(\mathcal O_{\mathbb P^m}\oplus \mathcal O_{\mathbb P}(-1)^{\oplus r})$. Taking a complete intersection of $r+1$ divisors in $[D_0]$ gives a subvariety $P$ of $X_{r,m}$, which is the zero section of the projection $\pi: X_{r,m}\rightarrow \mathbb P^m$. As stated in \cite{D}, the divisor at infinity $D_\infty$ does not intersect the zero section $P$.

The Calabi ansatz can be described as follows. Let $\omega_{FS}$ be the Fubini study metric on $\mathbb P^m$. Let $h$ be the Hermitian metric on $\mathcal O_{\mathbb P^m}(-1)$ such that $Ric(h)=-\omega_{FS}$, which means $h^{\oplus(r+1)}$ is a metric on $E$. Choose a local trivialization of $E$, which we write as $\nu=(\nu_1,...,\nu_{r+1})$, and let $(z_1,z_2,...,z_m)$ be inhomogeneous coordinates on $\mathbb P^m$. Consider the real variable $\rho$,  which in this trivialization can be expressed via
\be
\label{rhodef}
e^\rho=h(z)|\nu|^2=(1+|z|^2)|\nu|^2.
\ee
This leads to the following question: For a real valued function $\mu(\rho):\mathbb R\rightarrow\mathbb R$, what conditions on   $\xi_1\in\mathbb R$ and $\mu(\rho)$  guarantee that
$$ \omega:=\xi_1\omega_{FS}+\frac{i}{2\pi}\partial\bar\partial \mu(\rho)$$
defines a K\"ahler metric on $X_{r,m}$? To answer this,   Calabi demonstrated in \cite{Ca2} that $\omega$ is a well defined K\"ahler metric in the class
$$\omega\in \xi_1[D_H]+b[D_\infty]$$
if and only if $\xi_1$ and $b$ are positive, the function $\mu$ lies in $M_{0,b}$,  and in addition $\mu$ satisfies  $\frac{d\mu}{d\rho}>0,  \frac{d^2\mu}{d\rho^2}>0$, as well as  
 \be
 \label{gassump}
\qquad\frac{d}{dr}\hat \mu_{-\infty}(0)>0\qquad {\rm and}\qquad\frac{d}{dr}\hat \mu_{ \infty}(0)>0. \nonumber
 \ee
Here the space $M_{0,b}$ and the functions $\hat\mu_{-\infty}(r)$ and $\hat\mu_\infty(r)$ are defined in Section \ref{functionspace}.

Additionally, for the choice of real numbers $\xi_2$ and $q$, we can define a $(1,1)$ form $\alpha$ in the class $\xi_2[D_H]+q[D_\infty]\in H^{1,1}(X_{r,m},\mathbb R)$ via
$$ \alpha:=\xi_2\omega_{FS}+\frac{i}{2\pi}\partial\bar\partial g(\rho),$$
provided $g$ is in the space $M_{0,q}$. Here we have no assumptions on the positivity of derivatives of $g$ or the positivity of $\xi_2$ and $q$, since the class $[\alpha]$ is  arbitrary and need not be K\"ahler.

As mentioned in the introduction, the dHYM equation seeks a representative of the class $[\alpha]$ for which the the top dimensional form $(\omega+i\alpha)^{m+r+1}$ has constant argument. In particular the equation can be written as
\be
\label{dHYMCalabi}
{\rm Im}(e^{-i\hat\theta}(\omega+i\alpha)^{m+r+1})=0,
\ee
where $e^{-i\hat\theta}$ is a fixed constant. Integrating the above equation we see that the angle $\hat\theta$ must be the argument of the complex number
$$\zeta_{r,m}:=\int_{X_{r,m}}(\omega+i\alpha)^{m+r+1}.$$
By the $\partial\bar\partial$-Lemma $\zeta_{r,m}$ is independent of the choice of representative from $[\omega]$ or $[\alpha]$. Thus we see a simple necessary class condition for existence to a solution to \eqref{dHYMCalabi} is that $\zeta_{r,m}\neq0$.

We now seek a solution to \eqref{dHYMCalabi} by looking among representatives of $[\omega]$ and  $[\alpha]$ that satisfy the Calabi ansatz. Given $\omega$ and $\alpha$, if $\lambda_1,...,\lambda_{m+r+1}$ denote the real eigenvalues of the Hermitian endomorphism $\omega^{-1}\alpha$, then at a point where $\omega^{-1}\alpha$ is diagonal we can rewrite \eqref{dHYMCalabi} as
\be
\label{dhymeigen}
{\rm Im}\left(e^{-i\hat\theta}\frac{(\omega+i\alpha)^{m+r+1}}{\omega^{m+r+1}}\right)={\rm Im}\left(e^{-i\hat\theta}\prod_{k=1}^{m+r+1}(1+i\lambda_k)\right)=0.
\ee
In the case that    both $\omega$ and $\alpha$ satisfy the Calabi ansatz, it is shown in \cite{FL}  that the eigenvalues of $\omega^{-1}\alpha$ are
\be \underbrace{\frac{\xi_2+g'}{\xi_1+\mu'},...,\frac{\xi_2+g'}{\xi_1+\mu'}}_\text{m-{\rm times}}\,,\,\underbrace{\frac{g'}{\mu'},...,\frac{g'}{\mu'}}_\text{r-{\rm times}}\,,\,\frac{g''}{\mu''}.\nonumber
\ee
Here, as in Section \ref{functionspace}, we have let $'$ denote a derivative with respect to the variable $\rho$. Again denote $\mu'$ by $x$, which is strictly increasing and can be viewed as a coordinate on $(0,b)$. Also $g'$ defines a function $ f(x)=g'$, and taking a derivative in $\rho$ we see 
$$g''=\frac{d}{d\rho}f(x)=\frac{df}{dx}\frac{dx}{d\rho}= \frac{df}{dx}\mu''.$$
Thus we can re-write the eigenvalues as 
\be \underbrace{\frac{\xi_2+f}{\xi_1+x},...,\frac{\xi_2+f}{\xi_1+x}}_\text{m-{\rm times}}\,,\,\underbrace{\frac{f}{x},...,\frac{f}{x}}_\text{r-{\rm times}}\,,\,\frac{df}{dx}.\nonumber
\ee
Plugging into \eqref{dhymeigen} allows us to rewrite the dHYM equation as 
\be
{\rm Im}\left(e^{-i\hat\theta}\left(1+i\frac{\xi_2+f}{\xi_1+x}\right)^m\left(1+i\frac{f}x\right)^r\left(1+i\frac{df}{dx}\right)\right)=0.\nonumber
\ee
This is an ODE for a real function $f:[0,b]\rightarrow \mathbb R$.

Next we see this ODE is exact. Multiplying  by $(\xi_1+x)^mx^r$ gives
\be
\label{intermediateform}
{\rm Im}\left(e^{-i\hat\theta} (\xi+x+if )^m (x+if )^r\left(1+i\frac{df}{dx}\right)\right)=0,
\ee
where we have set $\xi=\xi_1+i\xi_2$. Now, define the complex polynomial
\be
w'(z):=e^{-i\hat\theta} (\xi+z )^m z^r,\nonumber
\ee
and let $w(z)$ be the antiderivative of $w'(z)$ satisfying $w(0)=0$. Then \eqref{intermediateform} can be written as
\be
{\rm Im}\left(w'(x+if)\left(1+i\frac{df}{dx}\right)\right)=\frac{d}{dx}{\rm Im} (w(x+if)  )=0,\nonumber
\ee
and we see that the dHYM equation in this setting is equivalent to finding a function $f:[0,b]\rightarrow\mathbb R$ whose graph lies on the level set of a harmonic polynomial $v={\rm Im}(w(z))$.

\begin{lemma}
\label{levelsetboundary}
The boundary points $z_1=0$ and $z_2=b+iq$ both lie on the same level set $\mathcal C_0$ of $v$. 
\end{lemma}
\begin{proof}
By construction we have $w(z_1)=0$. Thus to complete the lemma we need to show ${\rm Im}(w(z_2))=0$. As a first step we  compute the complex number $\zeta_{r,m}$.
Equation (2.6) in \cite{SY} gives that
\bea
\label{volume}
\omega^{r+m+1}&=&(\xi_1+\mu')^m\frac{h^{r+1}}{e^{(r+1)\rho}}(\mu')^r\mu''\left(\omega_{FS}^m\wedge\prod_{k=1}^{r+1}\frac i{2\pi}d\nu^k\wedge d\bar\nu^k\right)\nonumber\\
&=&\frac{(\xi_1+\mu')^m(\mu')^r\mu''}{|\nu|^{2(m+1)}}\left(\omega_{FS}^m\wedge\prod_{k=1}^{r+1}\frac i{2\pi}d\nu^k\wedge d\bar\nu^k\right).\nonumber
\eea
Switching to polar coordinates in the fiber, with radius $R$ and $dS$ the spherical volume form, we see
\bea
\omega^{r+m+1}&=&\frac{(\xi_1+\mu')^m(\mu')^r\mu''}{R^{2(m+1)}}R^{2m+1}dR\wedge dS\wedge \omega_{FS}^m\nonumber\\
&=&(\xi_1+\mu')^m(\mu')^r\mu''d{\rm log}(R)\wedge dS\wedge \omega_{FS}^m.\nonumber
\eea
For simplicity we denote  $dS\wedge \omega_{FS}^m$ by $dV$. Using  equation \eqref{rhodef} we have $d\rho=d{\rm log}h+2d{\rm log}R$, and because the base direction is already saturated by $\omega_{FS}^m$, we arrive at
\be
\omega^{r+m+1}=\frac{(\xi_1+\mu')^m(\mu')^r\mu''}2d\rho\wedge dV=\nonumber\frac{(\xi_1+x)^mx^r}2 dx\wedge dV,\nonumber
\ee
where we used $x=\mu'(\rho)$. Let $P'(z)=(\xi+z)^mz^r$, and set $P(z)$ to be the antiderivative with $P(0)=0$. Furthermore, let $M$ denote the unit sphere bundle over $\mathbb P^m$ induced by $E\rightarrow \mathbb P^m$. Then 
\bea
\zeta_{r,m}&=&\int_{X_{r,m}}\left(1+i\frac{\xi_2+f}{\xi_1+x}\right)^m\left(1+i\frac{f}x\right)^r\left(1+i\frac{df}{dx}\right)\omega^{r+m+1}\nonumber\\
&=&\frac12\int_{X_{r,m}}\left(\xi+x+if\right)^m\left(x+i{f}\right)^r\left(1+i\frac{df}{dx}\right)dx\wedge dV\nonumber\\
&=&\frac12\int_{X_{r,m}}\frac{d}{dx}P(x+if)\,dx\wedge dV\nonumber\\
&=&\frac12{\rm Vol}(M)(P(b+iq)-P(0))=\frac12{\rm Vol}(M) P(z_2).\nonumber
\eea
This gives $e^{i\hat\theta}=\frac{P(z_2)}{|P(z_2)|}$. Now $w(z_2)=e^{-i\hat\theta}P(z_2)$, and so 
$${\rm Im}\left(w(z_2)\right)={\rm Im}\left(e^{-i\hat\theta}P(z_2)\right)={\rm Im}\left(\frac{\overline{P(z_2)}}{|P(z_2)|}P(z_2)\right)=0.$$

\end{proof}
We now aim to apply our previous results to the polynomial $w(z)$ defined above. Note that  $w(z)$ has two critical points, $0$ and $-\xi$ (unless $r=0$ and then only $-\xi$ is a critical point), and since $\xi_1>0$ there are no critical points in ${\mathcal H}_x$. Thus   we can use Theorem \ref{Cauchystable} to determine precisely which initial classes $[\omega]=\xi_1[D_H]+b[D_\infty]$ and $[\alpha]=\xi_2[D_H]+q[D_\infty]$ admit a solution to the deformed Hermitian-Yang-Mils equation.  We remark that for this setting $\vartheta=\{0,z_2\}$, and  $\xi$ is not on the level set $\mathcal C_0$, but rather a critical point of $w(z)$. However,  the choice of $\xi$ determines $w(z)$ and thus the counting function $N$, so we find it appropriate to consider it as part of the initial data when determining if the setup is stable.

We would like to interpret Theorem \ref{Cauchystable}  in terms of the geometry of $X_{r,m}$. In the introduction we defined the   {\it charge} of an analytic subvariety $V\subset X$ in \eqref{centralcharge}, which equivalently can be expressed as:
$$Z_V([\alpha])=-\int_V(-i)^{{\rm dim}(V)}(\omega+i\alpha)^{{\rm dim}(V)}.$$
 Our goal is to compute the charges of two subvarities, $D_\infty$ and $P$ of $X_{r,m}$,   and then prove Theorem \ref{mainCalabi}.

First, because  $D_\infty$ does not intersect the zero section $P$, we see that $D_\infty^k\cdot D_H^{m-k}\cdot P=0$ for any $1\leq k\leq m$. Since $P$ is the complete intersection of $r+1$ divisors in $[D_0]=[D_\infty-D_H]$, we have
\be
\label{DinftyP}
D_\infty^k\cdot D_H^{m-k}\cdot (D_\infty-D_H)^{r+1}=0
\ee
 for any $1\leq k\leq m$. Furthermore, because $D_H$ is the pullback of a class on $\mathbb P^m$, we see $D_H^k=0$ for $k>m$. Using these facts we can now compute  
 \bea
 \label{thing10}
 \int_{D_\infty}(\omega+i\alpha)^{r+m}&=&\int_{D_\infty}\left(\xi_1D_H+bD_\infty+i(\xi_2D_H+qD_\infty)\right)^{r+m}\\
 &=&D_\infty\cdot (\xi D_H+z_2 D_\infty)^{r+m}\nonumber\\
 &=&D_\infty\cdot (z_2(D_\infty-D_H)+(z_2+\xi) D_H)^{r+m}.\nonumber
 \eea
 Now, by \eqref{DinftyP}, and the fact that $D_H^k=0$ for $k>m$, the only non-zero term in the above binomial expansion occurs when $(D_\infty-D_H)$ is raised to the power $r$ and $D_H$ is raised to the power $m$.  Thus 
 \bea
 \int_{D_\infty}(\omega+i\alpha)^{r+m}&=&z_2^r(z_2+\xi)^mD_\infty \cdot(D_\infty-D_H)^r\cdot D_H^m\nonumber\\
 &=&z_2^r(z_2+\xi)^mD_\infty^{r+1}  \cdot D_H^m,\nonumber
 \eea
 where the last line follows because $D_H$ was already saturated to the power $m$. But $D_\infty^{r+1}  \cdot D_H^m$ is just some real constant times the volume of $X_{r,m}$. Since we are only worried about the complex argument of various  terms we  take this to be one. Thus our above computation gives:
\be
\label{thing11}
\int_{D_\infty}(\omega+i\alpha)^{r+m}=z_2^r(z_2+\xi)^m=e^{i\hat\theta}w'(z_2).
\ee

Next we turn to the question of what our counting function $N(z_2)$ is  computing. Since $z_2$ is away from the $y$-axis $N(z_2)$ is determined by the Cauchy index ${\rm ind}_{-\infty}^{q}(R_{b})$ where $R_{b}$ is the rational function 
$$R_b(y)=\frac{{\rm Im}(w'(b+iy))}{{\rm Re}(w'(b+iy))}.$$
If we consider the  path $b+it:\mathbb R\rightarrow\mathbb C$ that follows the line $x=b$ from $y=-\infty$ up to $y=q$, the  Cauchy index ${\rm ind}_{-\infty}^{q}(R_{b})$ computes the number of times $w'(b+it)$ crosses the real axis. We remark that we used $t\in(-\infty,q]$ to compute the counting function $N(z_2)$ because of how we chose to number the sets $\{A_1,...,A_n\}$. For the current geometric application, however,  it makes more sense to look at the path $b+it$ for $t\geq q$, as we shall see.

Consider the path $Z_{D_\infty}(t):[q,\infty)\rightarrow\mathbb C$ defined via
\be
\label{Dinftylift}
Z_{D_\infty}(t)=-(-i)^{m+r}\int_{D_\infty}(\xi D_H+(b+it)D_\infty)^{m+r}.
\ee
By \eqref{thing10} we see $Z_{D_\infty}(q)=Z_{D_\infty}([\alpha])$. As $t$ approaches infinity, the term $(itD_\infty)^{m+r}$ dominates the binomial expansion, and since the  factor of $(-i)^{m+r}$ cancels with the $(i)^{m+r}$ from this term, we see that $Z_{D_\infty}(t)$ approaches the negative real axis as $t$ goes to infinity. Furthermore, by \eqref{thing11} we have
\be
\label{Dinfinitydef}
Z_{D_\infty}(t)=-(-i)^{m+r}e^{i\hat\theta}w'(b+it).
\ee
Since $w'(z)$ has no critical points in ${\mathcal H}_x$, the path $Z_{D_\infty}(t)$ never passes through the origin.   Now, in the proof of Lemma \ref{labelinglemma}, we saw that the sign of $R_b(y)$ always jumps from positive to negative at each discontinuity. This implies $Z_{D_\infty}(t)$ winds counterclockwise around the origin. Now, the Cauchy index ${\rm ind}_{q}^{\infty}(R_{b})$ counts hows many times ${\rm Re}(w'(b+it))=0$. Using \eqref{Dinfinitydef}, this is precisely  where ${\rm Re}(-i^{m+r}e^{-i\hat\theta}Z_{D_\infty}(t))=0$, and this happens  when $Z_{D_\infty}(t)$ crosses the line   $\mathbb R (-i)^{m+r+1}e^{i\hat\theta}$. Thus  $Z_{D_\infty}(t)$ crosses the line determined by $Z_X([\alpha])$ precisely ${\rm ind}_{q}^{\infty}(R_{b})$ times. This gives a well defined lift of the argument of  $Z_{D_\infty}([\alpha])$.

We see something similar for the subvariety $P$. To begin
\bea
\int_P(\omega+i\alpha)^m&=&\int_P\left(\xi_1D_H+bD_\infty+i(\xi_2D_H+qD_\infty)\right)^{m}\nonumber\\
&=&(D_\infty-D_H)^{r+1}\cdot(\xi D_H+z_2 D_\infty)^{m}.\nonumber
\eea
Again by \eqref{DinftyP}, any term where $(D_\infty-D_H)^{r+1}$ is paired with $D_\infty^k$ for $k\geq1$ vanishes. Thus 
\be
\label{thing12}
 \int_P(\omega+i\alpha)^m=\xi^m(D_\infty-D_H)^{r+1}\cdot D_H^m=\xi^m \,D_\infty^{r+1}\cdot D_H^m=\xi^m.
 \ee
Here we again used the normalization that $D_\infty^{r+1}\cdot D_H^m=1$. Now, unlike the previous case,  here we see $ \int_P(\omega+i\alpha)^m$ is not equal to $e^{-i\hat\theta}w'(0)$ (expect when $r=0$), so we have to change our path slightly.

Notice that the path $e^{i\hat\theta}w'( it)=(it)^r(it+\xi)^m$, for $t\in[0,\infty)$, begins at the origin, wrapping counterclockwise out and approaching the ray $\mathbb R_+(i)^{r+m}$ as $t$ approaches $\infty$. The factor of $t^r$ only scales the path (and does not affect the argument), so we can divide by $t^r$ and still compute the same winding number. We can also multiply by $(-i)^r$, to arrive at a path that starts at $\xi^m$, and winds around the origin  approaching the ray $\mathbb R_+(i)^{m}$. Using this, we now define
\be
\label{ZPlift}
Z_{P}(t)=-(-i)^m\left(\frac{(-i)^r}{t^r}e^{i\hat\theta}w'( it)\right).
\ee
We see $Z_{P}(0)=-(-i)^m\xi^m$, which by \eqref{thing12} is equal to $-(-i)^m\int_P(\omega+i\alpha)^m$. Thus $Z_{P}(0)=Z_P([\alpha])$, and as $t$ grows  $Z_{P}(t)$ winds counterclockwise around the origin,  crossing the line determined  by $Z_X([\alpha])$  (which just as above corresponds to where ${\rm Re}(w'(it))=0$) exactly ${\rm ind}_{0}^{\infty}(R_{0})$ times. Here we remark that if the critical point $0$ of $w(z)$ is non-generic, the Cauchy index contributes an extra factor of $\frac12$. So in this case choose a small $\epsilon>0$ so ${\rm Re}(w'(it))\neq0$ for $t\in(0,\epsilon)$, and instead use the Cauchy index ${\rm ind}_{\epsilon}^{\infty}(R_{0})$. In either case the path $Z_{P}(t)$ continues to wind around the origin, becoming asymptotic to the negative real axis. This gives a well defined lift for the argument of  $Z_{P}([\alpha])$.

We now have defined lifts of $Z_{P}([\alpha])$ and $Z_{D_\infty}([\alpha])$. We next see how to lift the average angle $\hat\theta$, which determines a lift of $Z_X([\alpha])$. Recall that in the proof of Lemma \ref{levelsetboundary} we demonstrated that $e^{i\hat\theta}=\frac{P(z_2)}{|P(z_2)|}$, where the polynomial $P(z)$ is the antiderivative of $P'(z)=(\xi+z)^mz^r$ with $P(0)=0$. Now, it should be noted that rotating $\hat\theta$ by $\pi$ will still yield a solution of \eqref{dHYMCalabi}, and so in searching for a lift we need to specify $\hat\theta$ in an interval of length $\pi$ rather than $2\pi$. We will let $\hat\Theta$ denote our lift.

Note that $z_2$ lies on a specific curve $\Gamma_k$ on the level set $\mathcal C_0$, which is unique in the region $A_k$. Along $\Gamma_k$ we know ${\rm Im}(w)={\rm Im}(e^{-i\hat\theta}P)=0$, and so the vector $P(\Gamma_k(t))$ always points in the direction of $e^{i\hat\theta}$. Now, if $f(x)$ is a function such that $x+if(x)$ lies on $\Gamma_k$, by \eqref{intermediateform} we see 
$$\hat\theta=m\,{\rm arctan}\left(\frac{\xi_2+f}{\xi_1+x}\right)+r\,{\rm arctan}\left(\frac{ f}{ x}\right)+{\rm arctan}\left(\frac{df}{dx}\right).$$
Since $\Gamma_k$ is asymptotic ray to a ray in the tangent cone $T^\infty_{\mathcal C_0}$ with slope $\psi\in(-\frac{\pi}2,\frac\pi2)$, as $x\rightarrow \infty$ we see arctan$\left(\frac fx\right)$ approaches $\psi$. Thus sending $x$ to infinity we arrive at the lift
\be
\label{firstlift}
\hat\Theta=(m+r+1)\psi.
\ee
Although we used the explicit analytic structure of the level set $\mathcal C_0$ to determine \eqref{firstlift}, it can also be done algebraically using the Cauchy index. 

By Proposition \ref{onecurveregion}, each region $A_k$ contains a unique curve $\Gamma_k$ from $\mathcal C_0$, including the top region $A_{m+r+1}$.  Suppose $\Gamma_{m+r+1}$ is asymptotic to a ray of angle $\psi_{top}$. Because  the rays that make up $T^\infty_{\mathcal C_0}$ have angle $\frac\pi{m+r+1}$ between them, we know the top angle must satisfy 
$$\frac{\pi}2>\psi_{top}\geq\frac{\pi}2-\frac\pi{m+r+1},$$
with equality on the right hand side occurring in the non-generic case. We then define $\hat\theta_{top}=(m+r+1)\psi_{top}$.  In fact, the determination of $\hat\theta_{top}$ does not need any analytic knowledge of the level set of $\mathcal C_0$, as it can simply be defined as the unique element of $\{\hat\theta+\ell \pi\,|\,\ell\in\mathbb Z\}$  that   lies in the interval   $[(m+r+1)\frac\pi2-\pi,(m+r+1)\frac\pi2)$. We now lift $\hat\theta$ by determining which region $A_k$ contains $z_2$. Since the Cauchy index   ${\rm ind}_{q}^{\infty}(R_{b})$ counts how many times the set $\{z\in\mathbb C\,|\,{\rm Re}(z)=b,\,{\rm Im}(z)>q\}$ crosses $\mathcal D_0$, we have that $z_2$ lies in $A_{m+r+1-{\rm ind}_{q}^{\infty}(R_{b})}$. Now, as we have stated the rays of $T^\infty_{\mathcal C_0}$ have angle $\frac\pi{m+r+1}$ between them. Because the angle associated to $P(\Gamma_k(t))$ is $(m+r+1)$ times the angle of the asymptotic ray, we arrive at the definition
\be
\label{secondlift}
\hat\Theta=\hat\theta_{top}-{\rm ind}_{q}^{\infty}(R_{b})\pi.
\ee
Thus we get a well defined lift of arg$(\zeta_{r,m})$.

Recall   the charge $Z_X([\alpha])= -(-i)^{m+r+1}\zeta_{r,m}$. Multiplying by $(-i)^{m+r+1}$ rotates a vector clockwise by $(m+r+1)\frac \pi2$, and given that $\hat\theta_{top}$  lies in the interval  $[(m+r+1)\frac\pi2-\pi,(m+r+1)\frac\pi2)$, it follows that $(-i)^{m+r+1}e^{i\hat\theta_{top}}$ lies in the lower half plane, and thus $-(-i)^{m+r+1}e^{i\hat\theta_{top}}$ lies in the upper half plane. We now use the path $Z_X(t):=-(-i)^{m+r+1}e^{i\hat\theta_{top}-t\pi}$, for $t$ in $[0,{\rm ind}_{q}^{\infty}(R_{b})]$, which starts in the upper half plane and wraps clockwise about the origin, and ending in the same direction of $Z_X([\alpha])$, to lift arg$(Z_X([\alpha]))$.

We have now constructed three paths in $\mathbb C^*$, denoted $Z_X(t)$ $Z_P(t)$, and $Z_{D_\infty}(t)$, which  we use to determine lifts of the arguments of $Z_X([\alpha])$, $Z_P([\alpha])$, and $Z_{D_\infty}([\alpha])$, respectively. Next we see how their arguments compare, and using Theorem \ref{Cauchystable} we are able to conclude with a proof of Theorem \ref{mainCalabi}.

We begin with  $Z_{D_\infty}(t)$. Note that because this path wraps counterclockwise in positive $t$, we start with $t=-\infty$ along the negative real axis, and as $t$ decreases use ${\rm ind}_{q}^{\infty}(R_{b})$ to count the number of times the path crosses the line determined by $Z_X([\alpha])$ as $t$ goes to zero from above. Since the first place that $Z_{D_\infty}(t)$ crosses this line is in the upper half plane along the ray $\mathbb R_+[ -(-i)^{m+r+1}e^{i\hat\theta_{top}}]$ (which is where the path $Z_X(t)$ starts), it is evident by construction that after wrapping around ${\rm ind}_{q}^{\infty}(R_{b})$ times we have
\be
\label{conclude1}
{\rm arg}(Z_X([\alpha]))<{\rm arg}(Z_{D_\infty}([\alpha]))< {\rm arg}(Z_X([\alpha]))+\pi.
\ee
We remark that this inequality only holds because our choice of how to lift ${\rm arg}(Z_X([\alpha]))$ depended on the region  containing $z_2$. 

Now we turn to $Z_P([\alpha])$, and see how its argument relates to existence of a solution to \eqref{dHYMCalabi}. Note in the generic case there are exactly ${\rm ind}_{0}^{\infty}(R_{0})$ curves from $\mathcal D_0$ crossing the positive $y$-axis. For a solution to exists, we need to find a level curve of $\mathcal C_0$ connecting $z_2$ to the origin. As we have discussed $z_2$ lies in $A_{m+r+1-{\rm ind}_{q}^{\infty}(R_{b})}$, thus if ${\rm ind}_{q}^{\infty}(R_{b})<{\rm ind}_{0}^{\infty}(R_{0})$ any path connecting $z_2$ to the origin must cross $\mathcal D_0$ in ${\mathcal H}_x$ and no solution exists. In the generic case there are  $r$ curves from  $\mathcal D_0$ emanating from the origin. Thus the condition
$${\rm ind}_{0}^{\infty}(R_{0})\leq  {\rm ind}_{q}^{\infty}(R_{b})\leq {\rm ind}_{0}^{\infty}(R_{0})+r$$
is equivalent to existence of a solution to  \eqref{dHYMCalabi} by Theorem \ref{Cauchystable} $(ii)$. Now, ${\rm ind}_{0}^{\infty}(R_{0})$ counts the number of times the path $Z_P(t)$ crosses the line determined by $Z_X([\alpha])$, while ${\rm ind}_{q}^{\infty}(R_{b})$ counts how many times the path $Z_X(t)$ starting at $-(-i)^{m+r+1}e^{i\hat\theta_{top}}$ travels distance $\pi$. Since ${\rm ind}_{0}^{\infty}(R_{0})\leq  {\rm ind}_{q}^{\infty}(R_{b})$ we must always have ${\rm arg}(Z_X([\alpha]))<{\rm arg}(Z_{P}([\alpha]))$ for a solution to exist. Additionally, since ${\rm ind}_{0}^{\infty}(R_{0})\geq  {\rm ind}_{q}^{\infty}(R_{b})-r,$ the most $Z_X(t)$ can continue to wind clockwise (lowering the angle) after the first time it passes $Z_P([\alpha])$ is $(r+1)\pi$. Thus a solution to  \eqref{dHYMCalabi} exists if and only if 
\be
\label{conclude2}
{\rm arg}(Z_X([\alpha]))<{\rm arg}(Z_{P}([\alpha]))< {\rm arg}(Z_X([\alpha]))+(r+1)\pi.
\ee

In the non-generic case, there are $r+1$ curves emanating from the origin, instead of $r$. In this case,  Theorem \ref{Cauchystable} $(iii)$ implies that a smooth  solution to \eqref{dHYMCalabi} exists if and only if 
$${\rm ind}_{0}^{\infty}(R_{0})+\frac12\leq  {\rm ind}_{q}^{\infty}(R_{b})\leq {\rm ind}_{0}^{\infty}(R_{0})+r-\frac12.$$
Above we noted that in the non-generic case, the Cauchy index ${\rm ind}_{\epsilon}^{\infty}(R_{0})={\rm ind}_{0}^{\infty}(R_{0})-\frac12$ counts the number of times $Z_P(t)$ crosses the line determined by $Z_X([\alpha])$. Thus we see $Z_X(t)$ must wind clockwise past $Z_P([\alpha])$ and then travel a further distance $\pi$, implying ${\rm arg}(Z_X([\alpha]))+\pi<{\rm arg}(Z_{P}([\alpha])) $. While as above,  $Z_X(t)$ can continue to wind clockwise (lowering the angle) after the first time it passes $Z_P([\alpha])$, of a distance $(r+1)\pi$. By these two observations we see a smooth solution to  \eqref{dHYMCalabi} exists if and only if 
\be
\label{conclude3}
{\rm arg}(Z_X([\alpha]))+\pi<{\rm arg}(Z_{P}([\alpha]))< {\rm arg}(Z_X([\alpha]))+(r+1)\pi.
\ee
Now, dividing equations \eqref{conclude1},  \eqref{conclude2}, and  \eqref{conclude3} by $\pi$, and using the definition of the {\it grade} of $V$ from the introduction
$$ \phi_V([\alpha])=\frac1\pi{\rm arg}(Z_V([\alpha])),$$
we are able to conclude the proof of Theorem \ref{mainCalabi}.

 \end{normalsize}

\end{document}